\documentclass[twoside,11pt]{article}

\usepackage{blindtext} 
\usepackage[preprint]{jmlr2e}

\usepackage{mathtools}
\usepackage{centernot}
\usepackage{dsfont}
\usepackage{customCommands}

\usepackage{algorithm}
\usepackage{algpseudocode}
\usepackage{xcolor}

\usepackage{enumitem}

\usepackage{hyperref}

\hypersetup{
  colorlinks,
  citecolor=blue,
  linkcolor=blue,
  urlcolor=blue}

\newlist{CustomList}{enumerate}{1} 
\setlist[CustomList,1]{label={A\arabic*}),ref={A\arabic*}} 

\newcounter{constnum}
\newcommand{\const}[1]{\refstepcounter{constnum}\label{#1}}
\newenvironment{proofof}[1]{\begin{proof}\textbf{of {#1}}}{\end{proof}}
\begin{document}

\title{Quantitative Convergence Analysis of Projected Stochastic Gradient Descent for Non-Convex Losses via the Goldstein Subdifferential}

\author{\name Yuping Zheng \email zhen0348@umn.edu \\
       \addr Department of Electrical and Computer Engineering\\
       University of Minnesota, Twin Cities\\
       Minneapolis, MN 55414, USA
       \AND
       \name Andrew Lamperski \email alampers@umn.edu \\
       \addr Department of Electrical and Computer Engineering\\
       University of Minnesota, Twin Cities\\
       Minneapolis, MN 55414, USA}




\maketitle

\begin{abstract}%
Stochastic gradient descent (SGD) is the main algorithm behind a large body of work in machine learning. In many cases, constraints are enforced via projections, leading to projected stochastic gradient algorithms. In recent years, a large body of work has examined the convergence properties of projected SGD for non-convex losses in asymptotic and non-asymptotic settings. Strong quantitative guarantees are available for convergence measured via Moreau envelopes. However, these results cannot be compared directly with work on unconstrained SGD, since the Moreau envelope construction changes the gradient. Other common measures based on gradient mappings have the limitation that convergence can only be guaranteed if variance reduction methods, such as mini-batching, are employed. 
This paper presents an analysis of projected SGD for non-convex losses over compact convex sets. Convergence is measured via the distance of the gradient to the Goldstein subdifferential generated by the constraints. Our proposed convergence criterion directly reduces to commonly used criteria in the unconstrained case, and we obtain convergence without requiring variance reduction.  
We obtain results for data that are independent, identically distributed (IID) or satisfy mixing conditions ($L$-mixing). In these cases, we derive asymptotic convergence and $O(N^{-1/3})$ non-asymptotic bounds in expectation, where $N$ is the number of steps. In the case of IID sub-Gaussian data, we obtain almost-sure asymptotic convergence and high-probability non-asymptotic $O(N^{-1/5})$ bounds. In particular, these are the first non-asymptotic high-probability bounds for projected SGD with non-convex losses. 
\end{abstract}


\section{Introduction}
This paper focuses on the analysis of projected stochastic gradient descent (SGD) for solving optimization problems of the form:
$$
\min_{x\in\cX} \bbE[f(x,\bz)] =\min_{x\in\cX}\bar{f}(x),
$$
where $\cX$ is a compact convex constraint set, $\bbE$ denotes the expected value over the random variable $\bz$, and $\bar f$ is a smooth, but possibly non-convex loss. 

Stochastic gradient descent and its variants have a plethora of applications in machine learning. See e.g. \citep{bottou2018optimization,mcmahan2013ad,koren2009matrix,koren2021advances,zinkevich2010parallelized, zinkevich2003online,goodfellow2016deep}. Projected SGD is commonly employed for stabilization and regularization in  machine learning and neural networks, \citep{bottou2018optimization}, though often under different names. For example,  the projection scheme is called ``reprojection" in \citep{goodfellow2016deep} and a specific variant is called ``max-norm regularization" in \citep{srivastava2014dropout}. 

\paragraph*{Related Work.}
Due to its practical significance, a large body of literature has examined projected SGD and generalized families of algorithms that include projected SGD. We review work on asymptotic convergence and non-asymptotic bounds for non-convex problems next.  

Asymptotic convergence for projected SGD with non-convex objectives has a long history, with proofs dating back to at least \citep{ermol1998stochastic,ermoliev2003solution}. More recent work on asymptotic properties of projected SGD and its generalizations, such as proximal gradients, includes \citep{davis2020stochastic,bianchi2022convergence,majewski2018analysis,nguyen2023stochastic,josz2024proximal,duchi2018stochastic,asi2019stochastic,asi2019importance,li2022unified}. These works, and the work of the present paper, are largely based on continuous-time approximation methods described in \citep{kushner2003stochastic,borkar2023stochastic,benaim2006dynamics}.

Non-asymptotic bounds in expectation, measured with respect to Moreau envelopes and related measures, are given for IID data, $\bz_k$, in \citep{davis2019stochastic,deng2021minibatch,zhu2023unified,gao2024stochastic,alacaoglu2020convergence,davis2025stochastic,fatkhullin2025can} and dependent data under mixing conditions in \citep{alacaoglu2023convergence}. Non-asymptotic bounds in expectation, measured special variants of the proximal gradient mapping are given in \citep{ghadimi2016mini,lan2024projected} with similar measures used in \citep{he2025non,xie2025tackling}.

We will show in Section~\ref{sec:convergence_criterion} that the Moreau envelope measure from \citep{davis2019stochastic} and subsequent works do not reduce to the gradient norm, $\|\nabla \bar{f}(x)\|$, in the unconstrained case, which is arguably the most common measure for non-convex unconstrained problems. In contrast, we will show that measures from \citep{ghadimi2016mini} and related works do reduce to $\|\bar{f}(x)\|$, but result in a non-shrinking term that can only be mitigated by variance reduction methods, such as mini-batching. 

For convex losses, the convergence theory for projected SGD is more mature, with overviews given in \citep{hazan2016introduction,shalev2014understanding}. 

Beyond projected SGD and generalizations, a variety of alternative methods for enforcing constraints in stochastic optimization have been proposed. These include penalty methods \citep{lin2022complexity,alacaoglu2024complexity}, Frank-Wolfe methods \citep{reddi2016stochastic,lacoste2016convergence}, and Lagrangian methods \citep{papadimitriou2025stochastic}.

\paragraph*{Contributions.}
We present an analysis of projected SGD with performance measured by distance of $-\nabla \bar{f}(x)$ to the Goldstein subdifferential, \citep{goldstein1977optimization}, associated with the constraints. Unlike Moreau envelope measures, our measure reduces to $\|\nabla \bar{f}(x)\|$ in the unconstrained case, and unlike the proximal gradient mapping measures from \citep{ghadimi2016mini}, we can show convergence without variance reduction / mini-batching. 

For IID and $L$-mixing data, $\bz_k$, we show that our proposed measure converges asymptotically to $0$ in expectation under stochastic approximation step size conditions. For fixed step sizes, we give a non-asymptotic bound in expectation of $O(N^{-1/3})$, where $N$ is the number of steps. Currently, our bound is weaker than the $O(N^{-1/2})$ bound obtained with respect to the Moreau envelope in \citep{davis2019stochastic}. More work is required to determine if this is due to a fundamental difference in the measures, or a limitation of the current analysis.

For IID sub-Gaussian data, we show that our measure converges asymptotically to $0$ with probability $1$ under stochastic approximation step size conditions. For fixed step sizes, we give a non-asymptotic bound of $O(N^{-1/5})$, which holds with high probability. In particular, these are the first non-asymptotic high probability bounds for projected SGD with non-convex losses.

\section{Problem Setup} \label{sec:setup}

\subsection{Notation and terminology}

$\bbN$ denotes non-negative integers and $\bbR$ denotes the real numbers. Random variables are denoted in bold. If $\bx$ is random variable, then $\bbE[\bx]$ denotes its expected value. $\|x\|$ denotes the Euclidean norm over $\bbR^n$. The probabilistic indicator function is denoted by $\indic$. (The indicator function from variational / convex analysis will be denoted by $\cI_{\cX}$ below.) $\bbP$ denotes probability measure. If $\cF$ and $\cG$ are $\sigma$-algebras, then $\cF \lor \cG$ denotes the $\sigma$-algebra generated by the union of $\cF$ and $\cG$.

$\Pi_{\cX}(y)$ denotes the projection of $y$ onto a convex set $\cX$, i.e. $\Pi_{\cX}(y) = \argmin_{x \in \cX} \|y -x\|$. The Euclidean distance of $y$ to the set $\cX$ is denoted by $\dist(y,\cX)$.  

The boundary of $\cX$ is deonoted as $ \partial\cX$, the normal cone of $\cX$ at a point $x$ is denoted by $\cN_{\cX}(x)$, the tangent cone of $\cX$ at a point $x$ is denoted by $T_{\cX}(x)$. $\cN_{\cX}(x) = \{ \phi \vert \phi^\top x \ge \phi^\top z, \; \forall z \in \cX \}$. $T_{\cX}(x) = \{ t(y-x)\vert  y \in \cX, \; t \ge 0\}$.

Let $osc(\bar{f})$ denote the oscillation of a bounded function $\bar{f}$, which is defined by $osc(\bar{f}) = \sup_{x, x' \in \cX} |\bar{f}(x) - \bar{f}(x')|$.


\subsection{Projected SGD}

Assume that the initial value of $\bx_0 \in \cX$ is independent of $\bz_i$ for all $i \in \bbN$. Projected SGD is the algorithm: 
\begin{align} \label{eq:projected_SGD}
   \bx_{k+1} = \Pi_{\cX} \left(\bx_k  - \alpha_k \nabla_x f(\bx_k, \bz_k)  \right)
\end{align}
where $\alpha_k$ is the step size. Our main result holds for any determinitic step size sequence with $0<\alpha_k\le \frac{1}{2}$. We also describe special cases of constant step size, $\alpha_k=\alpha$, and standard stochastic approximation conditions:
\begin{equation}
  \sum_{k=0}^\infty \alpha_k =\infty , \hspace{30pt} \sum_{k=0}^\infty \alpha_k^2 < \infty.
\end{equation}

\subsection{Approximate Stationarity via the Goldstein Subdifferential}
The Goldstein subdifferential is a relaxed version of the Clarke subdifferential and is widely used in nonsmooth optimization. It was first introduced in \citep{goldstein1977optimization} and has been used for measuring the stationarity for optimization algorithms, e.g. \citep{davis2022gradient,zhang2020complexity}.

Let $\cX$ denote a closed convex set. If $\cI_{\cX}$ is the corresponding convex indicator function:
$$
\cI_{\cX}(x)=\begin{cases}
  0 & x\in \cX \\
  +\infty & x\notin \cX.
\end{cases}
$$
then the Clarke subdifferential reduces to the standard convex subdifferential, and corresponds to the normal cone:
$$
\overline{\partial} \cI_{\cX}(x)=\partial\cI_{\cX}(x)=\cN_{\cX}(x).
$$
See \cite{rockafellar2009variational} for details on these definitions. 

For $\epsilon >0$, the Goldstein subdifferential is defined in terms of the Clarke subdifferential by:
$$
\overline{\partial}_{\epsilon}g(x)=\mathrm{conv}\left(\bigcup_{\|y-x\|\le \epsilon} \overline{\partial}g(y) \right).
$$
Thus, in the simple case that $g=\cI_{\cX}$, we have 
$$
\overline{\partial}_{\epsilon}\cI_{\cX}(x)= \mathrm{conv}\left(\bigcup_{\|y-x\|\le \epsilon} \cN_{\cX}(y) \right).
$$

The standard first-order necessary optimality conditions give that if $x$ is a local minimizer of $\bar{f}$ then $-\nabla \bar f(x)\in \cN_{\cX}(x)$. This occurs if and only if $\dist(-\nabla \bar{f}(x),\overline{\partial}\cI_{\cX}(x))=0.$ In this work, we will bound the relaxed stationarity measure, $\dist(-\nabla\bar f(x),\overline{\partial}_{\epsilon}\cI_{\cX}(x))$. 

\subsection{$L$-mixing processes}

In this paper, we consider the case that the external data variables, $\bz_k$,  can have dependencies over time, but these dependencies satisfy a property known as $L$-mixing. 
The class of $L$-mixing processes was introduced in \citep{gerencser1989class} and has been used to quantify the time-correlation in stochastic optimization in recent years \citep[see][]{barkhagen2021stochastic,chau2019fixed,chau2021stochastic,
zheng2022constrained,zheng2025non, zheng2025nonasymptotic}. It contains a wide variety of processes including measurements of geometrically ergodic Markov chain \citep{gerencser2002new}, which is suitable to model various of stable nonlinear stochastic systems. Furthermore, the class of $L$-mixing processes is closed under a variety of operations. In particular, $L$-mixing random variables results in another $L$-mixing sequence after passing through a stable, causal linear filter \citep{zheng2025non}. Therefore, the class of $L$-mixing processes contains a wide variety of data streams from system identification and time-series analysis.

Now we introduce the definition of the discrete-time $L$-mixing processes.
Let $\cF_k$ be an increasing family of $\sigma$-algebras and let $\cF^+$ be a decreasing family of $\sigma$-algebras such that $\cF_k$ and $\cF^+_k$ are independent for all $k \ge 0$. 
A discrete-time stochastic process $\bz_k$ is called $L$-mixing with respect to $(\cF, \cF^+)$ if 
\begin{itemize}
  \item $\bz_k$ is $\cF_k$-measurable for all integers $k \ge 0$
  \item $\cM_m(\bz) \coloneqq \sup_{k \ge 0 } \bbE^{1/m} \left[\|\bz_k\|^m\right] < \infty$ for all $m\ge 1$
  \item $\Psi_m(\bz) \coloneqq \sum_{\tau = 0}^\infty \psi_m(\tau, \bz) < \infty$ for all integers $k \ge 1$ and all $m\ge 1$,\\
  where $\psi_m(\tau, \bz) = \sup_{k \ge \tau} \bbE^{1/m}\left[\left\| \bz_k - \bbE[\bz_k \vert \cF^+_{k -\tau}]\right\|^m \right]$.
\end{itemize}

The value of $\Psi_m(\bz)$ measures how fast the time-dependence between data decays. 
\subsection{Assumptions}

\paragraph{General Assumptions.}

For the rest of the paper, $\cX$ denotes a compact convex subset of $\bbR^n$ of diameter $D$ which contains a ball of radius $r>0$ around the origin. 
Assume that for each $z$, $\nabla_x f(x,z)$ is $\ell$-Lipschitz in both $x$ and $z$, i.e. $\|\nabla_x f(x_1, z) - \nabla_x f(x_2, z)\| \le \ell \|x_1 - x_2\|$ and $\|\nabla_x f(x, z_1) - \nabla_x f(x, z_2)\| \le \ell \|z_1 - z_2\|$. This implies that $\| \nabla \bar{f}(x_1) - \nabla \bar{f}(x_2) \| \le \ell \|x_1 - x_2\|$, $\|\nabla \bar{f}(x)\| \le u $ where $u \le \nabla \bar{f}(0) + \ell D$ as well as $osc(\bar{f}) \le  D u$.

Note that without further specification in the paper, we simply use $\nabla f(x,z)$ to indicate $\nabla_x f(x,z)$.

\noindent
\paragraph{Assumptions on the external random variables $\bz_k$.}

In this work, we present the convergence bound under different assumptions on the external random variables $\bz_k \in \cZ$:

\begin{CustomList}
  \item $\nabla f(x, \bz_k)=\nabla \bar{f}(x)+\bz_k$,  where $\bz_k$ are IID zero mean sub-Gaussian random vectors, independent of the initial state, $\bx_0$. Specifically,  there exists a number $\hat{\sigma} > 0$ such that for all $v \in \bbR^n$, the following bound holds: \label{assumption:sub-Gaussian} 
  \begin{align} \label{eq:sub-Gaussian}
    \bbE\left[e^{v^\top \bz}\right] \le e^{\frac{1}{2} \hat{\sigma}^2 \|v\|^2}.
  \end{align}    
  \item $\bbE\left[\|\nabla f(x, \bz_k) - \nabla \bar{f}(x)\|^2\right] \le \sigma^2$ \label{assumption:bounded_var} and $\bz_k$ are independent for all $k \in \bbN$.
  \item $\bz_k$ is $L$-mixing processes , independent of the initial state, $\bx_0$.\label{assumption:L-mixing} 
\end{CustomList}

Note that \ref{assumption:sub-Gaussian} is a special case of both \ref{assumption:bounded_var} and  \ref{assumption:L-mixing}. Indeed, using that $\bbE[(e_i^\top \bz)^2]\le \hat\sigma^2$ for each standard basis vector, $e_i$, gives that $\bbE[\|\bz\|^2]\le n \hat\sigma^2$. To see that \ref{assumption:L-mixing} holds,  we can set $\cF_k = \sigma\left(\{\bz_0,\ldots,\bz_k\}\right)$ and $\cF_k^+=\sigma\left(\{\bz_{k+1},\bz_{k+2},\ldots\}\right)$. Then we can bound the moments via bounds on the moment generating function, noting that for all $m\ge 1$: $\Psi_m(\bz)=\psi_m(0,\bz)$. In particular, $\Psi_2(\bz)\le \sqrt{n}\hat \sigma$. 

\section{Approximation and Main Results} \label{sec:main_results}

In this section, we present the continuous-time approximation of the algorithm via ordinary differential equations (ODEs).
Then, we present the main results under our proposed convergence criterion. More discussion on convergence criteria is shown in Section \ref{sec:convergence_criterion}.

\subsection{Continuous-Time Approximation} \label{sub:preliminaries}
The following lemma is the key to the application of ODE method to approximate the discrete-time processes with continuous-time processes.

\begin{lemma} \label{lem:proj_tangent_cone}
  For  all $x \in \cX$, $g \in \bbR^n$, the following holds:
  \begin{align*}
    \lim_{\alpha \downarrow 0} \frac{\Pi_{\cX}(x +\alpha g) - x}{\alpha} = \Pi_{T_{\cX}(x)} (g)
  \end{align*} 
\end{lemma} 
This result appears in \citep{calamai1987projected, mccormick1972gradient} and the corresponding proof can be found in Proposition 2 of \citep{mccormick1972gradient}. 

Lemma \ref{lem:proj_tangent_cone} implies that projected SGD can be viewed as a constrained stochastic Euler approximation to the following ODE:
\begin{align} \label{eq:x_C}
  \frac{d}{dt} \bx_t^C = \Pi_{T_{\cX}(\bx_t^C)}(- \nabla \bar{f}(\bx_t^C)).
\end{align}
Note $\bx^C$ is called the \textit{continuous} process in the rest of the paper.

Let $\tau_k = \sum_{j=0}^{k-1}\alpha_j$, which measures the total amount of continuous time that has been simulated prior to the computation of $\bx_k$. 
To analyze projected SGD in terms of continuous-time processes, we let $\bx_t^A$ denote the iterates of \eqref{eq:projected_SGD} embedded into continuous-time as:
$$
\bx_{t}^A = \bx_k \quad \textrm{if } t\in [\tau_k,\tau_{k+1}).
$$

\begin{figure}[ht]
\centering
\begin{minipage}[b]{.45\textwidth}
  \centering
\includegraphics[width=\textwidth]{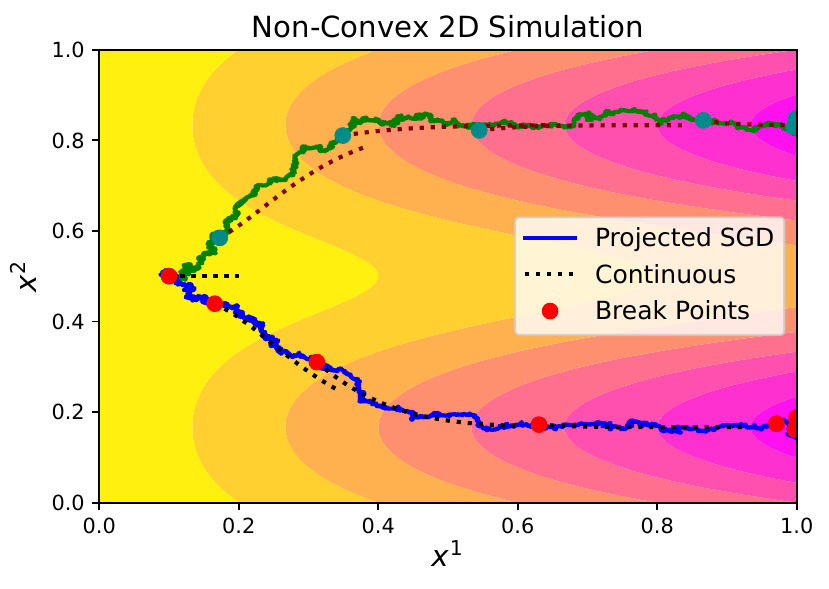}
\end{minipage}
\begin{minipage}[b]{.45\textwidth}
  \centering
\includegraphics[width=.95\textwidth]{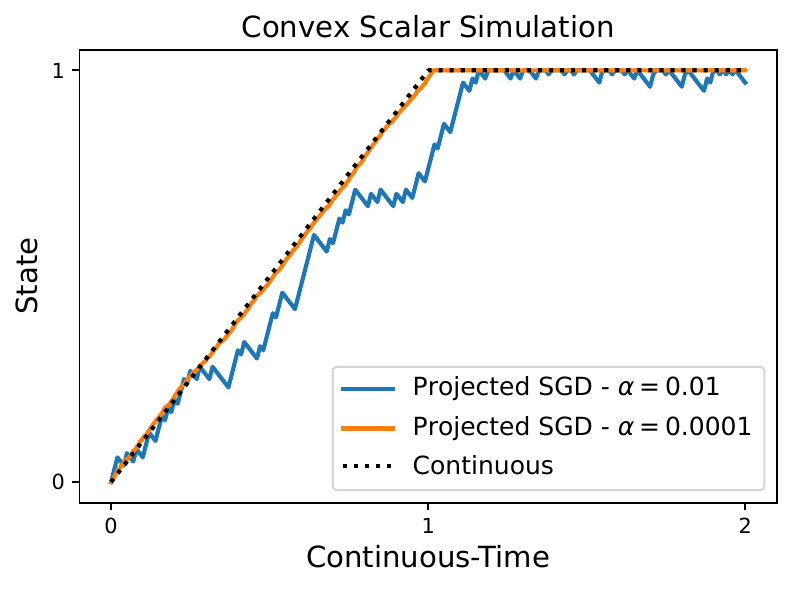}
\end{minipage}
\caption{\label{fig:simulations} {\bf Simulations.} The left shows two runs of projected SGD for a non-convex system from the same starting point. The combination of stochaticity and non-convexity implies that two trajectories with the same starting point can diverge over time. Here, the solid lines show the result of projected SGD, the dotted lines show the continuous-time approximations, and the filled circles indicate the break points. The right shows two runs of projected SGD on a convex scalar problem. With small step size, $\alpha=0.0001$, the trajectory converges to a small region near the optimal solution. However, the existing convergence measures for constrained problems amplify the small flucuations. } 
\end{figure} 

As seen in Fig.~\ref{fig:simulations}, projected SGD and its continuous-time approximation can drift apart due to instabilities. So, for our convergence analysis, we will construct a sequence of restarted continuous-time processes, defined as follows.

For a fixed number of iterates, $N$, define break points by:
\begin{align*}
  s_0 &= 0 \\
  s_{i+1}&=  \max\{\tau_j \vert \tau_j - s_i \le 1, 0 \le \tau_j \le \tau_N \} \textrm{ if } s_i<\tau_N.
\end{align*}

Then, for 
 $t \in [s_i, s_{i+1}]$, set:
\begin{align*}
  \frac{d}{dt} \bx_t^{C_i} &= \Pi_{T_{\cX}(\bx_t^{C_i})}(- \nabla \bar{f}(\bx_t^{C_i}))\\
  \bx_{s_i}^{C_i} &= \bx_{s_i}^A.
\end{align*}

For compact notation, define $\bx_t^J$ to be the process that jumps between the continuous processes: $\bx_t^J=\bx_t^{C_{i}}$ when $t\in [s_i,s_{i+1})$.

For $k\ge 0$, let 
$$
\boldb_k = \sup_{t\in [\tau_k,\tau_{k+1})}\|\bx_t^J-\bx_{\tau_k}^A\|.
$$

Denote $\chi(N) = \max\{ i \vert  s_i < \tau_N \}$ so that  $s_{\chi(N)+1} =\tau_N$.  Then the total number of subintervals partitioning the interval $[0,\tau_N]$ is $\chi(N)+1$. Let $\cK(i)$ denote the value of $j$  such that $ \tau_j = s_i$, and let $\zeta(j)$ denote the value of $i$ such that $\cK(i)\le j<\cK(i+1)$.

Assume that $\alpha_k \le \frac{1}{2}$ for all $k \in [0,N-1]$. Then for any $i \in [0, \chi(N)-1]$, there exists $j \in [0, N-1]$ s.t. $s_{i+1} =\tau_j \le s_i +1 \le \tau_{j+1} = \tau_j + \alpha_j \le \tau_j + \frac{1}{2}$, which implies that $s_i + \frac{1}{2} \le s_{i+1} \le s_i +1 $, i.e. $\frac{1}{2} \le s_{i+1}- s_i \le 1$. The last interval is $[s_{\chi(N)}, s_{\chi(N)+1}]$ whose length is at most $1$, but is  not necessarily greater than $\frac{1}{2}$. 

Figure~\ref{fig:main}  shows the partitions of the interval $[0,\tau_N]$ for constant step size and diminishing step size according to the construction rules above. For constant step size, set $\alpha = \frac{3}{8}$ and $N=9$, the interval $[0,\tau_N]$ is partitioned into 5 subintervals. For diminishing step size, set $\alpha_k = \frac{1}{k+2}$ and $N=20$, the interval $[0,\tau_N]$ is partitioned into 3 subintervals.

\begin{figure}[ht]
  \centering

  \begin{minipage}{0.9\textwidth}
      \centering
      \includegraphics[width=\linewidth]{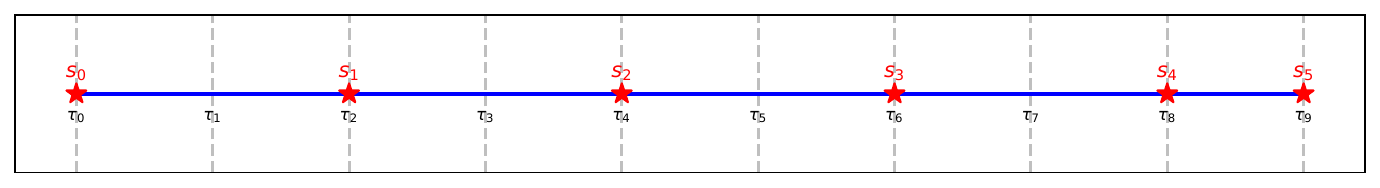}
      \par{(a) Constant Step Size}
  \end{minipage}

  \vspace{0.3cm} 

  \begin{minipage}{0.9\textwidth}
      \centering
      \includegraphics[width=\linewidth]{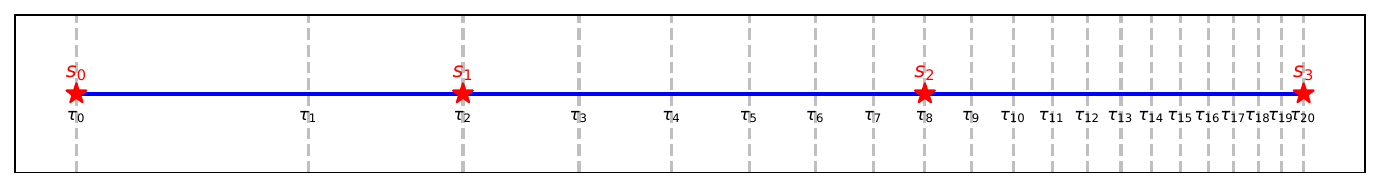}
      \par{(b) Diminishing Step Size}
  \end{minipage}

  \caption{Demonstration of the construction of subintervals $[s_i, s_{i+1}]$.}
  \label{fig:main}
\end{figure}

In the results below, we will use the following constants:
\const{const_sum_sq}
\const{const_max_sq_root}

\const{const_sum_sq_high_prob}
\const{const_max_high_prob_1}
\const{const_max_high_prob_2}
\const{const_general_1}
\const{const_general_2}
\begin{subequations}
  \label{eq:constants}
\begin{align}
  c_{\ref{const_sum_sq}}&=\begin{cases}
    e^{\ell} \sqrt{n} \hat\sigma & \textrm{Under Assumption }\ref{assumption:sub-Gaussian} \\
    e^{\ell} \sigma & \textrm{Under Assumption }\ref{assumption:bounded_var} \\
    2 \ell e^{\ell} \Psi_2(\bz) & \textrm{Under Assumption } \ref{assumption:L-mixing}
  \end{cases} \\
    c_{\ref{const_max_sq_root}}&= \left( u + \sqrt{2 r^{-1}u (Du+D^2)} \right)e^{\ell} \\
    c_{\ref{const_sum_sq_high_prob}}&= 2\sqrt{2}e^{2\ell} \hat\sigma D  \\
    c_{\ref{const_max_high_prob_1}}&= 4e^{2\ell}\hat\sigma^2 \\
    c_{\ref{const_max_high_prob_2}}&=e^{2\ell}(n+1) \hat \sigma^2 
\end{align}
\end{subequations}

\begin{align*}
\end{align*}

The following lemma gives bounds in expectation and with high probability on the deviations of the algorithm from the jumping continuous process $\|\bx_t^A-\bx_t^J\|$. It is proved in Appendix~\ref{app:supporting_lemmas}.

\begin{lemma} \label{lem:x_A2C_b}
  Assume that $0 < \alpha_k \le \frac{1}{2}$ for all $k\in \bbN$. 
  Let $\cK(i)$ be the sequence of integers defined in Section~\ref{sub:preliminaries}. The following hold:

  \begin{enumerate}[label=(\roman*)]
    \item If assumption \ref{assumption:sub-Gaussian}, \ref{assumption:bounded_var}, or \ref{assumption:L-mixing} holds, then for all integers $k\in [\cK(i),\cK(i+1))$:
  \begin{align*}
   &\bbE\left[ \boldb_k \right] \le c_{\ref{const_sum_sq}}\sqrt{\sum_{j={\cK(i)}}^{k-1} \alpha_j^2}  
   + c_{\ref{const_max_sq_root}}\max_{j \in [\cK(i),k]} \sqrt{ \alpha_j}. 
  \end{align*}
\item If Assumption \ref{assumption:sub-Gaussian} holds and $\delta\in (0,1)$, then with probability at least $1-  \delta$, 
  \begin{align*}
    &\max_{k\in [\cK(i),\cK(i+1))}\boldb_k \le 
    \left( c_{\ref{const_sum_sq_high_prob}} \sqrt{\log(2\delta^{-1})}\sqrt{ \sum_{j=\cK(i)}^{\cK(i+1)-1} \alpha_j^2}  + \left( c_{\ref{const_max_high_prob_1}} \log(2\delta^{-1}) + c_{\ref{const_max_high_prob_2}}\right)  \sum_{j=\cK(i)}^{\cK(i+1)-1} \alpha_j^2 \right)^{1/2}
   \\
   & \hspace{100pt} 
    +c_{\ref{const_max_sq_root}}\max_{j \in [\cK(i),\cK(i+1))} \sqrt{ \alpha_j} =: h_i(\delta) .
  \end{align*}
  \end{enumerate}
\end{lemma}

The next result shows that in the decaying step size case, the algorithm, $\bx_t^A$ converges to the jumping continuous process, $\bx_t^J$, asymptotically. Note that $\max_{k\in [\cK(i),\cK(i+1))}\boldb_k = \sup_{t\in [s_i,s_{i+1})}\|\bx_t^A-\bx_t^J\|$. 

\begin{proposition}
  \label{prop:convergence}
  Assume that $0<\alpha_k \le \frac{1}{2}$, $\sum_{k=0}^{\infty}\alpha_k=\infty$, and $\sum_{k=0}^{\infty}\alpha_k^2 <\infty$. 
  Let $h_i$ be the bounding function defined in Lemma~\ref{lem:x_A2C_b}. Set $\delta_i = \frac{\sum_{j=\cK(i)}^{\cK(i+1)-1}\alpha_j^2}{\sum_{k=0}^{\infty}\alpha_k^2}$. Then $\lim_{i\to\infty} h_i(\delta_i)=0$, and with probability $1$, the event
$$
\sup_{t\in [s_i,s_{i+1})}\|\bx_t^A-\bx_t^{J}\|> h_i(\delta_i)
$$
occurs at most finitely many times. In particular, $\lim_{t\to\infty}\|\bx_t^A-\bx_t^J\|=0$ with probability $1$.
\end{proposition}

\subsection{Main Results}

Here we present the main results of the paper. All of the results in this section are proved in Appendix \ref{app:main_results}.

\begin{theorem} \label{thm:convergence_general}
  Assume that $0 <\alpha_k \le \frac{1}{2}$ for all integers $k\in [0,N-1]$. Let $\chi(N)$ and $\cK(i)$ be the integers defined in Section~\ref{sub:preliminaries}.
  \begin{itemize}
    \item  If Assumption~\ref{assumption:sub-Gaussian}, \ref{assumption:bounded_var}, or \ref{assumption:L-mixing} holds, then
  \begin{equation*} \label{eq:general_bound}
  \begin{aligned}
    &\frac{1}{\tau_N} \sum_{k =0}^{N-1} \alpha_k \bbE\left[\dist\left(-\nabla \bar{f}(\bx_k),\overline{\partial}_{\boldb_k} \cI_{\cX}(\bx_k) \right)^2\right] \\ 
    &\le 
    \frac{1}{\tau_N} \sum_{i=0}^{\chi(N)} \left(c_{\ref{const_general_1}} \sqrt{\sum_{j=\cK(i)}^{\cK(i+1)-1} \alpha_j^2} + c_{\ref{const_general_2}} \max_{j \in [\cK(i),\cK(i+1))} \sqrt{\alpha_j}\right) + \frac{Du}{\tau_N},
  \end{aligned}
\end{equation*}
where
$$
c_{\ref{const_general_1}}=(u+2u\ell)c_{\ref{const_sum_sq}}  \quad\textrm{and}\quad
    c_{\ref{const_general_2}}=(u +2 u \ell) c_{\ref{const_max_sq_root}}.
$$
\item If Assumption~\ref{assumption:sub-Gaussian} holds, then for any collection of numbers $\delta_0,\ldots,\delta_{\chi(N)}$ such that $0 < \delta_i$ and $\sum_{i=0}^{\chi(N)}\delta_i < 1$, with probability at least $1-\sum_{i=0}^{\chi(N)}\delta_i$, the following bound holds:
  \begin{equation*} \label{eq:general_bound_hp}
  \begin{aligned}
    &\frac{1}{\tau_N} \sum_{k=0}^{N-1} \alpha_k \dist\left(-\nabla \bar{f}(\bx_k),\overline{\partial}_{h_{\zeta(k)}(\delta_{\zeta(k)})} \cI_{\cX}(\bx_k) \right)^2 
    \le \frac{u+2\ell u}{\tau_N} \sum_{i=0}^{\chi(N)}h_i(\delta_i)
+ \frac{Du}{\tau_N}.
  \end{aligned}
\end{equation*}
\end{itemize}
\end{theorem}

\begin{remark} \label{rem:unconstrained_generalization}
  The convergence criterion in Theorem~\ref{thm:convergence_general} generalizes the common sum of norm square convergence criterion for unconstrained SGD. 
  In the unconstrained case, $\overline{\partial}_{\boldb_k}\cI_{\cX}(\bx_k)=\{0\}$, 
 which gives
\begin{equation} \label{eq:sum_of_squares}
\begin{aligned}[b]
  \frac{1}{\tau_N} \sum_{k=0}^{N-1}
\alpha_k \bbE\left[\dist\left(-\nabla \bar{f}(\bx_k),\overline{\partial}_{\boldb_k} \cI_{\cX}(\bx_k) \right)^2\right]  & = \frac{1}{\tau_N} \sum_{k=0}^{N-1} \alpha_k \bbE\left[\| \nabla \bar{f}(\bx_{\tau_k}^A) \|^2\right].
\end{aligned}
\end{equation}
In particular, when $\alpha_k \equiv \alpha $, then
$\eqref{eq:sum_of_squares} = \frac{1}{N}\sum_{k=0}^{N-1} \bbE\left[\| \nabla \bar{f}(\bx_{\tau_k}^A) \|^2\right]. $ In both the  variable or constant step size cases, \eqref{eq:sum_of_squares} matches convergence criteria for non-convex functions in \citep{bottou2018optimization}.

\end{remark}

\begin{corollary} \label{cor:convergence_Robbin_Monro}
  Assume that $0<\alpha_k \le \frac{1}{2} $ for all integers $k \ge 0$, $\sum_{k=0}^{\infty} \alpha_k = \infty$ and $\sum_{k=0}^\infty \alpha_k^2 <\infty$.
  \begin{itemize}
    \item If Assumption~\ref{assumption:sub-Gaussian}, \ref{assumption:bounded_var}, or \ref{assumption:L-mixing} holds, then
  \begin{align*} \label{eq:convergence_asymptotic}
        \lim_{N \rightarrow \infty} \frac{1}{\tau_N} 
        \sum_{k=0}^{N-1}\alpha_k\bbE\left[\dist\left(-\nabla \bar{f}(\bx_k),\overline{\partial}_{\boldb_k} \cI_{\cX}(\bx_k) \right)^2\right] 
        =0 \quad\textrm{and}\quad \lim_{k\to\infty}\bbE[\boldb_k]=0.
  \end{align*}
\item If Assumption~\ref{assumption:sub-Gaussian} holds, then with probability $1$
  $$
 \lim_{N \rightarrow \infty} \frac{1}{\tau_N} 
        \sum_{k=0}^{N-1}\alpha_k\dist\left(-\nabla \bar{f}(\bx_k),\overline{\partial}_{\boldb_k} \cI_{\cX}(\bx_k) \right)^2
        =0 \quad \textrm{and} \quad \lim_{k\to\infty}\boldb_k = 0.
  $$
\end{itemize}
\end{corollary}

\const{const_const_step_1}
\const{const_const_step_2}
\begin{corollary} \label{cor:convergence_constant_step}
  Assume that $0<\alpha_k=\alpha \le \frac{1}{2} $ for all integers $k \in[0,N-1]$. 
  \begin{itemize}
    \item If Assumption~\ref{assumption:sub-Gaussian}, \ref{assumption:bounded_var}, or \ref{assumption:L-mixing} holds, then
  \begin{equation}\label{eq:convergence_constant}
    \hspace{-30pt}
    \frac{1}{\tau_N}
\sum_{k=0}^{N-1}\alpha_k\bbE\left[\dist\left(-\nabla \bar{f}(\bx_k),\overline{\partial}_{\boldb_k} \cI_{\cX}(\bx_k) \right)^2\right] 
\le c_{\ref{const_const_step_1}} \sqrt{\alpha} + \frac{ c_{\ref{const_const_step_2}}}{N} \alpha^{-1} \:\textrm{ and }\:\bbE[\boldb_k]\le (c_{\ref{const_sum_sq}}+c_{\ref{const_max_sq_root}})\sqrt{\alpha},
  \end{equation} 
where the constants are given by
  \begin{equation*}
    c_{\ref{const_const_step_1}} = 2(c_{\ref{const_general_1}}+c_{\ref{const_general_2}}) \quad  
    \textrm{and}\quad c_{\ref{const_const_step_2}} = Du + c_{\ref{const_general_1}}+c_{\ref{const_general_2}} . 
  \end{equation*}
  In particular, if $\alpha = O(N^{-2/3})$ then both bounds in \eqref{eq:convergence_constant} are of $O(N^{-1/3})$.
%
\item If Assumption~\ref{assumption:sub-Gaussian} holds, then for any $\delta \in (0,1)$, with probability at least $1-\delta$:
  \begin{align}
    \label{eq:convergence_constant_hp}
    \hspace{-20pt}
\frac{1}{N}\sum_{k=0}^{N-1} \dist\left(-\nabla \bar{f}(\bx_k),\overline{\partial}_{q\left(\frac{\delta}{2\alpha N+1}\right)\alpha^{1/4}} \cI_{\cX}(\bx_k) \right)^2
\le   2 q\left(\frac{\delta}{2\alpha N+1}\right)\alpha^{1/4}+ \frac{q\left(\frac{\delta}{2\alpha N+1}\right)+Du}{\alpha N},
               \end{align}
               where
               $$
  q(\hat \delta)= (u +2 u \ell) \left( c_{\ref{const_sum_sq_high_prob}} \sqrt{ \log(2\hat\delta^{-1})}  + \left( c_{\ref{const_max_high_prob_1}} \log(2\hat \delta^{-1}) + c_{\ref{const_max_high_prob_2}}\right)   \right)^{1/2}
                 + (u +2 u \ell)c_{\ref{const_general_2}}.
               $$
               In particular if $\alpha = O(N^{-4/5})$, then the bound in \eqref{eq:convergence_constant_hp} is of $O\left(N^{-1/5}\sqrt{\log\left( N^{1/5}\delta^{-1}\right)}\right)$.
             \end{itemize}
\end{corollary}

\section{Discussion on Convergence Criteria} \label{sec:convergence_criterion}
In this section, we review various convergence criteria used for analyzing gradient-descent algorithms under different hypotheses.


For GD and projected GD algorithms with convex objectives, we can use $\bar{f}(x_k) - \bar{f}(x^*)$ to measure the convergence rate since all critical points, $x^*$, are actually global minima \citep{boyd2004convex, bubeck2015convex,nesterov2018lectures}. In the strongly convex case, $\|x_k - x^*\|^2$ is often used to measure the convergence \citep{nesterov2018lectures}, since minimizers are unique. A stochastic variation $\bbE[\bar{f}(\bx_k) - \bar{f}(x^*)]$ is used under the conditions that $\bar{f}$ is non-strongly convex and global minimum exists (not necessarily unique) \citep{moulines2011non} or if $\bar{f}$ satisfies the Polyak-Lojasiewicz condition \citep{khaled2020better, gower2021sgd}. The stochastic version $\bbE[\|\bx_k - x^*\|^2]$ is used when $\bar{f}$ is strongly convex for both unconstrained and projected SGD \citep{moulines2011non}.

For non-convex problems, algorithms may converge to critical points which are not necessarily global minima. In general, there could be multiple critical points. So, measures based on $\bar{f}(x)-\bar{f}(x^*)$ or $\|x-x^*\|$ with fixed critical points, $x^*$, will not be suitable. In asymptotic analysis, it is common to measure convergence of the algorithms to the set of critical points \citep{bianchi2022convergence,ermol1998stochastic,ermoliev2003solution}. 

For non-asymptotic bounds for non-convex problems, most analyses utilize variations on the size $\|\nabla \bar{f}(x)\|$ to measure stationarity.
For example, in unconstrained deterministic problems, \citep{nesterov2018lectures} uses $\min_{0 \le k < N } \|\nabla \bar{f}(x_k)\|$. For stochastic problems, $\min_{0 \le k <N}\bbE[ \|\nabla \bar{f}(\bx_k)\|]$ is used in \citep{khaled2020better,yuan2022general,lei2019stochastic, wu2020finite}, 
\\$\frac{1}{N} \bbE[\sum_{k=0}^{N-1} \|\nabla \bar{f}(\bx_k)\|^2]$ is used for constant step sizes in \citep{bottou2018optimization, zhang2020global, chen2023finite}, and $ \bbE[\frac{1}{\tau_N}\sum_{k=0}^{N-1} \alpha_k \|\nabla \bar{f}(\bx_k)\|^2]$ is used for diminishing step sizes in \citep{bottou2018optimization}. 
For a more thorough review of unconstrained SGD, see 
\citep{garrigos2023handbook}.

The most common measure for non-asymptotic analysis of projected SGD and its generalizations is based on Moreau envelopes. See \citep{davis2019stochastic,deng2021minibatch,zhu2023unified,gao2024stochastic,alacaoglu2020convergence,davis2025stochastic,fatkhullin2025can}. For $\lambda > 0$, the Moreau envelope and proximal map of a function $\psi:\bbR^n\to\bbR\cup\{+\infty\}$ is defined respectively by:
$$
\psi_{\lambda}(x)=\min_{y\in \bbR^n} \left(\psi(y)+\frac{1}{2\lambda}\|x-y\|^2 \right) \quad\textrm{and}\quad \mathrm{prox}_{\lambda \psi}(x)=
\argmin_{y\in \bbR^n} \left(\psi(y)+\frac{1}{2\lambda}\|x-y\|^2 \right).
$$
For projected SGD, the function $\psi=\bar{f}+\cI_{\cX}$ is used in Moreau envelope analysis. 

The gradient of the Moreau envelope is given by $\nabla \psi_{\lambda}(x)=\frac{1}{\lambda}\left(x-\mathrm{prox}_{\lambda \psi}(x) \right)$.
In \citep{davis2019stochastic} and subsequent work, convergence of projected SGD is measured via
$$
\frac{1}{\tau_N}\sum_{k=0}^n\alpha_k \bbE\left[\|\nabla \psi_{\lambda}(\bx_k) \|^2 \right],
$$
where $\lambda > 0$ is a fixed number with bounds scaling with $\lambda^{-1}$. 

While the Moreau envelope measure resembles the common sum-of-squared norms measure from unconstrained SGD, it does not reduce to the  value in the unconstrained case. Indeed, if 
$$
\psi(x)=\bar{f}(x)=\frac{1}{2}x^\top Px + q^\top x,
$$
with positive definite $P$, 
then $\nabla \bar{f}(x)=Px+q$ and 
$\nabla \psi_{\lambda}(x)=(\lambda P+I)^{-1}(P x + q)$. For more complex objectives, the relationship between $\nabla \bar{f}$ and $\nabla \psi_{\lambda}$ will be more complex. These differences make direct comparison of Moreau envelope results with work on unconstrained SGD challenging.


An alternative measure, proposed in \citep{ghadimi2016mini} and used later in \citep{lan2024projected} is 
$$
\bbE \left[ \frac{1}{\alpha_\br^2} \left\|\bx_{\br} - \Pi_\cX\left(\bx_{\br} -  \alpha_k \frac{1}{m_{\br}} \sum_{i=1}^{m_{\br}}\nabla f(\bx_{\br}, \bz_{\br,i})\right)\right\|^2\right]
$$
where $\br$ is a randomly drawn iteration and $\{\bz_{\br,1},\ldots,\bz_{\br,m_{\br}}\}$ is a minibatch of noise variables. (Note that the convex projection here is a special case covered by their theory.)

The convergence measures in \citep{ghadimi2016mini,lan2024projected} are modifications of the reduced gradient described in \citep{nesterov2018lectures}. Related measures are commonly used in deterministic settings. In projected GD, the measure $\frac{1}{\alpha_k} \|x_k - \Pi_\cX(x_k -  \alpha_k \nabla \bar{f}(x_k))\|$ ($\alpha_k=1$) is used in \citep{royer2024full},  the measure $\text{dist}\left(- \nabla \bar{f}(x_k), \cN_{\cX}(x_k)\right)$ is used in \citep{olikier2025projected}, while $\|T_{\cX(x_k)}(-\nabla \bar{f}(x_k))\|$ is used in \citep{di2024stationarity,calamai1987projected, balashov2022error}.

The example below shows that it is impossible to achieve low error with respect to the measure from \citep{ghadimi2016mini} and related measures, unless the variance of the randomness is reduced. As a result, to achieve low error, \citep{ghadimi2016mini,lan2024projected} propose large  mini-batches. 
Similar limitations appear  in the work of \citep{he2025non,xie2025tackling}.

\begin{example} \label{example}
Let $f(x,z) = -x+xz$ so that $\nabla_x f(x,z)=-1 + z$. Set the constraint to be $\cX = [-1,1]$. Let $\bz_k$ follows the scaled binary Rademacher distribution such that $\bbP(\bz_k = 2) = 0.5$ and $\bbP(\bz_k = -2) = 0.5$.

The normal cone of $\cX$ is given by:
\begin{align*}
  \cN_{\cX}(x) = \begin{cases}
  0 & x \in (-1,1) \\
  (-\infty, 0] & x = -1 \\
  [0, \infty) & x= 1.
  \end{cases}
\end{align*}

Projected SGD becomes
\begin{align*}
  \bx_{k+1} = \Pi_{\cX} \left( \bx_k + \alpha_k (1 -\bz_k)\right).
\end{align*}

Note that $\nabla \bar{f}(x)=-1$ for all $x$. Furthermore, for all $y\in\cX$, 
\begin{equation}
  \label{eq:discontinuousDistance}
\dist(-\nabla\bar{f}(x),\cN_{\cX}(y))=\begin{cases}
  1 & y\in [-1,1) \\
  0 & y=1.
\end{cases}
\end{equation}

Say that $0<\alpha_k< \frac{1}{2}$ and $\alpha_{k+1}\le \alpha_k$. Then for any $\bx_k \in \cX$,  we  have $\bx_{k+1} \in (-1, 1-\alpha_{k+1}]$ with probability at least $\frac{1}{2}$.  Thus, for all $k\ge 0$, with probability at least $\frac{1}{2}$, we have 
\begin{equation*}
  \frac{1}{\alpha_k}\left\|\bx_k-\Pi_{\cX}\left(\bx_k-\alpha_k\nabla f(\bx_k,\bz_k)\right) \right\|= |1-\bz_k| \ge 1 
\end{equation*}
and 
\begin{gather*}
  \frac{1}{\alpha_{k+1}}\left\|\bx_{k+1}-\Pi_{\cX}(\bx_{k+1}-\alpha_{k+1}\nabla\bar{f}(\bx_{k+1})) \right\| = \\
  \dist(-\nabla \bar{f}(\bx_{k+1}),\cN_{\cX}(\bx_{k+1})) \!= \!\|\Pi_{T_{\cX}(\bx_{k+1})}(-\nabla \bar{f}(\bx_{k+1}))\|
  =1.
\end{gather*}
So, the average of any of these criteria will be at least $\frac{1}{2}$.
\end{example}

While these common convergence metrics remain bounded away from zero, on average, Fig.~\ref{fig:simulations} (along with the theory in this paper) shows that the projected SGD solutions closely follow the continuous-time trajectory, $\bx_t^C$, when the step size is small. The issue is that these measures amplify small random  fluctuations near the boundary. 


\section{Conclusion and Future Work} \label{sec:conclusion}

In this work, we gave a new convergence analysis of projected SGD where stationarity is measured by the distance of the gradient from the Goldstein subdifferential generated by the constraints. This proposed convergence measure allows direct comparison with results on unconstrained problems and does not require variance reduction techniques to achieve convergence. Our results hold in expectation for both IID and mixing data sequences, giving both asymptotic convergence and non-asymptotic bounds. In the special case of IID data sequences, we obtain asymptotic convergence almost surely and give the first non-asymptotic high probability bounds.    

Future work is needed to clarify the relation of our results and prior work. In particular, tighter bounds are achieved with respect to the Moreau envelope in \citep{davis2019stochastic}, and it would be useful to understand if this is due to fundamental differences in the measure or limitations of our analytic technique. 
Extensions of the work include the analysis of adaptive step size rules, as commonly arise in applications, or 
incorporation into more complex algorithmic schemes, such as policy gradient algorithms within actor-critic reinforcement learning algorithms. 


\bibliography{cool-refs}

\appendix

\section{Convergence Analysis via Intermediate Processes}  \label{sec:convergence_intermediate}

In this section, we introduce some intermediate processes to bound $\bbE\left[\left\|\bx_t^A - \bx_t^C\right\| \right]$. The bound of $\bbE\left[\left\|\bx_t^A - \bx_t^C\right\| \right]$ and  all the supporting lemmas are shown in Appendix~\ref{app:supporting_lemmas}. The ideas of using the intermediate processes and the proofs on the quantitative bounds are similar to those in \citep{lamperski2021projected,zheng2022constrained}.

Other than the \textit{continuous} time process $\bx_t^C$ defined in \eqref{eq:x_C}, we further introduce the \textit{mean} process, $\bx_{\tau_k}^M$, and the \textit{discretized} process, $\bx_{\tau_k}^D$ with $\bx_{\tau_{k_0}}^A = \bx_{\tau_{k_0}}^M = \bx_{\tau_{k_0}}^C = \bx_{\tau_{k_0}}^D$ where integers $k_0 \le k$:
\begin{align} 
    \bx_{\tau_{k+1}}^M = \Pi_{\cX} \left(\bx_{\tau_k}^M - \alpha_k \nabla \bar{f}(\bx_{\tau_k}^M)\right).
\end{align}

The discretized processes $\bx_{\tau_k}^D$ uses the form of Skorokhod solution introduced in Appendix \ref{app:skorokhod}. 

Here is some preliminary to define $\bx_{\tau_k}^D$. To enable the construction of Skorokhod problem, which is key to prove Lemma~\ref{lem:x_C2D} relying on Lemma 2.2 (i) in \citep{tanaka1979stochastic}, we first show that the alternative representation of projected ODE \eqref{eq:x_C}:
\begin{align} \label{eq:x_C_Normal}
  \frac{d}{dt} \bx_t^C = - \nabla \bar{f}(\bx_t^C) - \bv_t^C
\end{align}
where $\bv_t^C \in \cN_{\cX}(\bx_t^C)$. 

The projected ODE \eqref{eq:x_C_Normal} in the context of constrained stochastic approximation can be found in \citep{kushner2003stochastic} and these two equivalent forms of projected ODE are also mentioned in \citep{borowski2025convergence} but there was no proof. The equivalence of \eqref{eq:x_C} and \eqref{eq:x_C_Normal} follows from the Moreau decomposition (e.g. \cite{hiriart2004fundamentals}), which implies that that for any vector $g \in \bbR^n$, $\Pi_{T_{\cX}(x)}(g) = g - \Pi_{\cN_{\cX}(x)}(g)$. In particular, $\bv_t^C = \frac{\Pi_{\cN(x)} (- \nabla \bar{f}(\bx_t^C))}{\|\Pi_{\cN(x)} (- \nabla \bar{f}(\bx_t^C))\|}$ when $\bx_t^C \in \partial \cX$.

Then, the projected ODE \eqref{eq:x_C_Normal} can be written as 
\begin{align} \label{eq:ODE_reflected}
  d \bx_{t}^C = - \nabla \bar{f}(\bx_{t}^C) dt - \bv_t^C d\bmu^C(t).
\end{align}
Here, $- \int_0^t \bv_t^C d\bmu^C(t)$ is a bounded variation reflection process that keeps $\bx_t^C \in \cX$ for all $t \in [0,\tau_N]$, as long as $\bx_0^C \in \cX$. The measure, $\bmu^C$, is non-negative and supported on $\{s \vert \bx_s^C \in \partial \cX\}$, while  $\bv_s^C \in \cN_{\cX} (\bx_s^C)$. With these conditions on \eqref{eq:ODE_reflected}, $\bv_t^C d\bmu^C(t)$ is uniquely defined and $\bx^C$ is the unique solution to the Skorokhod problem for a process defined below:
\begin{align}
    \by_{t}^C = \bx_0^C - \int_{0}^t \nabla \bar{f}(\bx_s^C) ds.
\end{align}

More details on Skorokhod problems are given in Appendix \ref{app:skorokhod}. 

In the following, we denote the Skorokhod solution for given trajectory, $\by$, by $\cS(\by)$.

Let  $\by_t^D = \by_{\tau_k}^C$ for all $t \in [\tau_k, \tau_{k+1})$. Such discretization operator is denoted by $\cD(\cdot)$. Then, we define $\bx^D = \cS(\cD(\by^C))$, i.e. $\bx_t^D = \by_{\tau_k}^C + \bphi_t^D$ for all $t \in [\tau_k, \tau_{k+1})$, where $\bphi_t^D = - \int_0^t \bv_t^C d\bmu^C(t)$.  

Therefore, we have
\begin{align}
\nonumber
    \bx_{\tau_{k+1}}^D &= \Pi_{\cX}\left(\bx_{\tau_k}^D + \by_{\tau_{k+1}}^C - \by_{\tau_{k}}^C \right) \\
    & =   \Pi_{\cX}\left(\bx_{\tau_k}^D - \int_{\tau_k}^{\tau_{k+1}} \nabla \bar{f}(\bx_t^C) dt \right).
\end{align}

The intermediate processes are used to bound the individual terms from the following triangle inequality:
\begin{align*}
  \|\bx_{t}^A - \bx_{t}^C\| \le \|\bx_{\tau_k}^A - \bx_{\tau_k}^M\| + \|\bx_{\tau_k}^M- \bx_{\tau_k}^D\| + \|\bx_{\tau_k}^D - \bx_{\tau_k}^C\| +\|\bx_{t}^C - \bx_{\tau_k}^C\|.
\end{align*}


\section{Proof of Main Results} \label{app:main_results}
The section presents the proofs of the main results. 

\noindent
\begin{proofof}{Theorem~\ref{thm:convergence_general}} 
\label{proof:theorem_general}

Firstly, we have
\begin{equation} \label{eq:monotonicity_x_C}
  \begin{aligned}[b]
  \bar{f}(\bx_{s_{i+1}}^C) - \bar{f}(\bx_{s_i}^C)  &=  \int_{0}^{\tau_N} \frac{d}{dt} \bar{f}(\bx_t^C) dt \\
  & = \int_0^{\tau_N} \nabla \bar{f}(\bx_t^C)^\top \Pi_{T_{\cX}(\bx_t^C)}\left(- \nabla \bar{f}(\bx_t^C)\right) dt \\
  & = - \int_0^{\tau_N} \left\|\Pi_{T_{\cX}(\bx_t^C)}\left(- \nabla \bar{f}(\bx_t^C)\right) \right\|^2 dt 
\end{aligned}
\end{equation}
where the first and second equalities use the fundamental theorem of calculus and the chain rule respectively and the last equality uses Lemma \ref{lem:proj_tangent_cone_2} in Appendix~\ref{app:variational_geometry}. 

  For one time interval, $[s_i, s_{i+1}]$, we have the following decomposition:
\begin{align*}
    \bar{f}(\bx_{s_{i+1}}^A) - \bar{f}(\bx_{s_i}^A) &= \bar{f}(\bx_{s_{i+1}}^A) - \bar{f}(\bx_{s_i}^A) -\left(\bar{f}(\bx_{s_{i+1}}^{C_i}) - \bar{f}(\bx_{s_i}^{C_i})\right) + \left(\bar{f}(\bx_{s_{i+1}}^{C_i}) - \bar{f}(\bx_{s_i}^{C_i})\right) \\
    & = \bar{f}(\bx_{s_{i+1}}^A) - \bar{f}(\bx_{s_{i+1}}^{C_i}) - \int_{s_{i}}^{s_{i+1}} \left\|\Pi_{T_{\cX}(\bx_t^{C_i})}\left(- \nabla \bar{f}(\bx_t^{C_i})\right)\right\|^2 dt 
  \end{align*}
  where the second equality uses $\bx_{s_i}^A = \bx_{s_i}^{C_i}$ and \eqref{eq:monotonicity_x_C}.

  Adding and subtracting $\Pi_{T_{\cX}(\bx_t^{C_i})}\left(- \nabla \bar{f}(\bx_t^A)\right)$ inside the norm of the second term on the RHS and rearranging gives
  \begin{align*}
      & \int_{s_i}^{s_{i+1}} \left\|\Pi_{T_{\cX}(\bx_t^{C_i})}\left(- \nabla \bar{f}(\bx_t^A)\right) + \Pi_{T_{\cX}(\bx_t^{C_i})}\left(- \nabla \bar{f}(\bx_t^{C_i})\right) - \Pi_{T_{\cX}(\bx_t^{C_i})}\left(- \nabla \bar{f}(\bx_t^A)\right) \right\|^2 dt  \\
      &= \bar{f}(\bx_{s_{i+1}}^A) - \bar{f}(\bx_{s_{i+1}}^{C_i}) - \left(\bar{f}(\bx_{s_{i+1}}^A) - \bar{f}(\bx_{s_i}^A)\right) \\
       \Rightarrow & \int_{s_i}^{s_{i+1}} \left( \left\|\Pi_{T_{\cX}(\bx_t^{C_i})}\left(- \nabla \bar{f}(\bx_t^A)\right)\right\| -\left\| \Pi_{T_{\cX}(\bx_t^{C_i})}\left(- \nabla \bar{f}(\bx_t^{C_i})\right) - \Pi_{T_{\cX}(\bx_t^{C_i})}\left(- \nabla \bar{f}(\bx_t^A) \right) \right\| \right)^2 dt  \\
       &\le \bar{f}(\bx_{s_{i+1}}^A) - \bar{f}(\bx_{s_{i+1}}^{C_i}) - \left(\bar{f}(\bx_{s_{i+1}}^A) - \bar{f}(\bx_{s_i}^A)\right) \\
      \Rightarrow & \int_{s_i}^{s_{i+1}} \left\|\Pi_{T_{\cX}(\bx_t^{C_i})}\left(- \nabla \bar{f}(\bx_t^A) \right) \right\|^2 dt \\
      &\le  \left(  \bar{f}(\bx_{s_{i+1}}^A) - \bar{f}(\bx_{s_{i+1}}^{C_i}) - \left(\bar{f}(\bx_{s_{i+1}}^A) - \bar{f}(\bx_{s_i}^A)\right)\right) \\
      &+ 2 \int_{s_i}^{s_{i+1}} \left\|\Pi_{T_{\cX}(\bx_t^{C_i})}\left(- \nabla \bar{f}(\bx_t^A)\right) \right\| \left\|  \Pi_{T_{\cX}(\bx_t^{C_i})}\left(- \nabla \bar{f}(\bx_t^{C_i})\right) - \Pi_{T_{\cX}(\bx_t^{C_i})}\left(- \nabla \bar{f}(\bx_t^A)\right)\right\| dt\\
      & \le u \left\|\bx_{s_{i+1}}^A - \bx_{s_{i+1}}^{C_i}\right\| - \left( \bar{f}(\bx_{s_{i+1}}^A) - \bar{f}(\bx_{s_i}^A)\right) + 2u \ell  \int_{s_i}^{s_{i+1}}  \left\|\bx_t^A - \bx_t^{C_i} \right\| dt
  \end{align*}
  where the right arrow uses the fact that $(\|a\| - \|b\|)^2 \le \|a+b\|^2$ for all $a, b \in \bbR^n$.  The last inequality uses the fact that $\bar{f}$ is $u$-Lipschitz and $\nabla \bar{f}$ is $\ell$-Lipschitz as well as the non-expansiveness of the convex projection.

 Summing over $\chi(N)+1$ terms gives
 \begin{equation} \label{eq:sum_integral_subintervals}
  \begin{aligned}
    & \sum_{i=0}^{\chi(N)} \int_{s_i}^{s_{i+1}} \left\|\Pi_{T_{\cX}(\bx_t^{C_i})}\left(- \nabla \bar{f}(\bx_t^A)\right) \right\|^2 dt \\
    & \le u \sum_{i=0}^{\chi(N)} \left\|\bx_{s_{i+1}}^A - \bx_{s_{i+1}}^{C_i} \right\| - \sum_{i=0}^{\chi(N)}  \left(\bar{f}(\bx_{s_{i+1}}^A) - \bar{f}(\bx_{s_i}^A)\right) + 2u \ell  \sum_{i=0}^{\chi(N)}\int_{s_i}^{s_{i+1}}  \left\|\bx_t^A - \bx_t^{C_i}\right\| dt \\
    & = u \sum_{i=0}^{\chi(N)} \left\|\bx_{s_{i+1}}^A - \bx_{s_{i+1}}^{C_i}\right\| 
    +\left(\bar{f}(\bx_0^A)-\bar{f}(\bx_{\tau_N}^A)\right)
    + 2u \ell  \sum_{i=0}^{\chi(N)}\int_{s_i}^{s_{i+1}}  \left\|\bx_t^A - \bx_t^{C_i} \right\| dt \\
    & \le u \sum_{i=0}^{\chi(N)} \left\|\bx_{s_{i+1}}^A - \bx_{s_{i+1}}^{C_i} \right\|  + 2u \ell  \sum_{i=0}^{\chi(N)}\int_{s_i}^{s_{i+1}}  \left\|\bx_t^A - \bx_t^{C_i} \right\| dt +osc(\bar{f}) 
  \end{aligned}
\end{equation}
where the equality uses a telescoping sum.


Now, we examine the expected value case, which holds for all of the assumptions. 

Lemma~\ref{lem:x_A2C} gives the bound below 
\begin{align*}
    \bbE\left[\left\|\bx_{s_{i+1}}^A - \bx_{s_{i+1}}^{C_i} \right\|\right] \le g_1(s_{i+1} - s_i) \sqrt{\sum_{ \{j \vert   s_i \le \tau_j < s_{i+1}\}} \alpha_j^2} + g_2(s_{i+1} -s_i) \max_{\{j \vert s_i \le \tau_j \le s_{i+1}\}} \sqrt{\alpha_j}
\end{align*}
where $g_1(q) = \sigma e^{\ell q} $ under Assumption \ref{assumption:bounded_var}, $g_1(q) = 2 \ell \Psi_2 (\bz) e^{\ell q} $ under Assumption \ref{assumption:L-mixing} and $g_2(q) = e^{\ell q} \left( u +  \sqrt{2 r^{-1}  u \left(q D u  + D^2 \right)   }  \right) $.

Therefore, we have 
\begin{equation} \label{eq:ave_sum}
\begin{aligned}
  &\frac{1}{\tau_N} \sum_{i=0}^{\chi(N)} \bbE\left[ \int_{s_i}^{s_{i+1}} \left\|\Pi_{T_{\cX}(\bx_t^{C_i})}\left(- \nabla \bar{f}(\bx_t^A)\right) \right\|^2 dt \right] \\
  & \le \frac{1}{\tau_N} \left(\sum_{i=0}^{\chi(N)} \left(u \bbE\left[  \left\|\bx_{s_{i+1}}^A - \bx_{s_{i+1}}^{C_i} \right\| \right] + 2 u \ell (s_{i+1} - s_i) \max_{j \in [s_i, s_{i+1}]} \bbE\left[  \left\|\bx_{j}^A - \bx_{j}^{C_i} \right\| \right] \right) + osc(\bar{f}) \right) \\
  & \le \frac{1}{\tau_N} \left(\sum_{i=0}^{\chi(N)}(u + 2 u \ell (s_{i+1} - s_i)) \left( g_1(s_{i+1} - s_i) \sqrt{\sum_{ \{j \vert   s_i \le \tau_j < s_{i+1}\}} \alpha_j^2} \right. \right.\\
  &\left. \left. \hspace{160pt} + g_2(s_{i+1} -s_i) \max_{\{j \vert s_i \le \tau_j \le s_{i+1}\}} \sqrt{\alpha_j}\right) + osc(\bar{f}) \right).
\end{aligned}
\end{equation}

Note that if $t\in [s_i,s_{i+1})$, we must have that $\bx_t^A=\bx_{\tau_k}^A=\bx_k$ for some integer $\cK(i)\le k < \cK(i+1)$. In this case, $\|\bx_t^A-\bx_t^{C_i}\|\le \boldb_k$, by the definition of $\boldb_k$.  

Lemma~\ref{lem:proj_tangent_normal}, followed by the definitions of the convex projection and the expression for the Goldstein subdifferential give 
\begin{align*}
  \left\|\Pi_{T_{\cX}(\bx_t^{C_i})}\left(- \nabla \bar{f}(\bx_t^A)\right)\right\| &= \left\|-\nabla \bar{f}(\bx_t^A)-\Pi_{\cN_{\cX}(\bx_t^{C_i})}(-\nabla \bar{f}(\bx_t^{A})\right\| \\
                                                                                  &= \dist\left(-\nabla \bar{f}(\bx_t^A),\cN_{\cX}(\bx_t^{C_i}) \right) \\
                                                                                  &\ge \dist\left(-\nabla \bar{f}(\bx_t^A),\overline{\partial}_{\|\bx_t^A-\bx_t^{C_i}\|}\cI_{\cX}(\bx_t^A)\right) \\
                                                                                  &\ge \dist(-\nabla \bar{f}(\bx_k),\overline{\partial}_{\boldb_k}\cI_{\cX}(\bx_k) ).
\end{align*}

It then follows that 
$$
\int_{\tau_k}^{\tau_{k+1}} \left\|\Pi_{T_{\cX}(\bx_t^{C_i})}\left(- \nabla \bar{f}(\bx_t^A)\right) \right\|^2 dt \ge \alpha_k  
\dist(-\nabla \bar{f}(\bx_k),\overline{\partial}_{\boldb_k}\cI_{\cX}(\bx_k) )^2.
$$
Plugging this lower bound into the integrals on the left of \eqref{eq:ave_sum} and using that $s_{i+1}-s_i\le 1$ gives the bound on the expected value.

Now, we turn to the special case that Assumption~\ref{assumption:sub-Gaussian} holds, and give a bound in high probability.

Plugging the definition of $\boldb_k$ into \eqref{eq:sum_integral_subintervals} and using the bound on the Goldstein subdifferentials above gives
\begin{align}
  \label{eq:goldstein_b_bound}
  \sum_{k=0}^{N-1} \alpha_k  
  \dist(-\nabla \bar{f}(\bx_k),\overline{\partial}_{\boldb_k}\cI_{\cX}(\bx_k) )^2 &\le
  (u+2u\ell)
  \sum_{i=0}^{\chi(N)} \max_{k\in [\cK(i),\cK(i+1)-1]}\boldb_k
+\osc(\bar{f}).
\end{align}
Applying Lemma~\ref{lem:x_A2C_b}, and using a union bound gives that with probability at least $1-\sum_{i=0}^{\chi(N)}\delta_i$, we have  $\boldb_k\le h_{\zeta(k)}(\delta_{\zeta(k)})$ for all $k=0,\ldots,N-1$.  Recall that $\zeta(k)$ was defined in Section~\ref{sub:preliminaries}. Using the bound $\boldb_k\le h_{\zeta(k)}(\delta_{\zeta(k)})$ on the left and right now gives the result.
\end{proofof}

\begin{proofof}{Corollary~\ref{cor:convergence_Robbin_Monro}}


Firstly, we know $s_{i+1} -s_i \ge \frac{1}{2}$ for all $i \in [0, \chi(N)-1]$. Then $\tau_N > s_{\chi(N)} = \sum_{i=0}^{\chi(N)-1} (s_{i+1} -s_i) \ge \frac{1}{2} \chi(N)$.
Furthermore, from the condition that $\sum_{k=0}^{\infty} \alpha_k^2 < \infty$, we have $\lim_{m \rightarrow \infty}\sum_{k = m}^\infty \alpha_k^2 = 0$. Therefore, $\lim_{i \rightarrow \infty} \sum_{ \{j \vert   s_{i} \le \tau_j < s_{i+1} \}} \alpha_j^2 = 0$. This implies that if we choose $ \epsilon >0$, there exists $i_1 \in \bbN$ such that for all $i \ge i_1$, $\sum_{\{j \vert s_i \le \tau_j \le s_{i+1}\}} \alpha_j^2 \le \epsilon^2$, so $\sqrt{\sum_{\{j \vert s_i \le \tau_j \le s_{i+1}\}} \alpha_j^2} \le \epsilon$. Since $\alpha_j \le \sqrt{\sum_{\{j \vert s_i \le \tau_j \le s_i\}} \alpha_j^2}$ for all $j$ such that $s_i \le \tau_j \le s_{i+1}$, then $\alpha_j \le \epsilon$ for all $j$ such that $s_i \le \tau_j \le s_{i+1}$ and $i \ge i_1$. Therefore, $\max_{\{j \vert s_i \le \tau_j \le s_{i+1}\}} \sqrt{\alpha_j} \le \sqrt{\epsilon}$ for $ i \ge i_1$. 

Without loss of generality, we can ignore the constant factors, since the right of \eqref{eq:ave_sum} is arbitrarily small, if in only if the following quantity is arbitrarily small: 
\begin{align*}
  & \frac{\sum_{i=0}^{\chi(N)} \left\{ \sqrt{\sum_{\{j \vert s_i \le \tau_j \le s_{i+1}\}} \alpha_j^2 }+ \max_{\{j \vert s_i \le \tau_j \le s_{i+1}\}} \sqrt{\alpha_j} \right\}}{s_{\chi(N)}} \\
   \le & \frac{\sum_{i=0}^{i_1} \left\{ \sqrt{\sum_{\{j \vert s_i \le \tau_j \le s_{i+1}\}} \alpha_j^2 }+ \max_{\{j \vert s_i \le \tau_j \le s_{i+1}\}} \sqrt{\alpha_j} \right\}}{\frac{1}{2} \chi(N)} + \frac{(\epsilon+ \sqrt{\epsilon} )(\chi(N) - i_1)}{\frac{1}{2} (\chi(N) -i_1)}.
\end{align*} 
The first term converges to zero as $\chi(N) \rightarrow \infty$ (i.e. $N \rightarrow \infty$) and the second term is $2(\epsilon + \sqrt{\epsilon})$, which is arbitrarily small. Therefore, we obtain asymptotic convergence for the expected value.

Now consider the case that Assumption~\ref{assumption:sub-Gaussian} holds. Equation~\ref{eq:goldstein_b_bound} implies that:
\begin{equation*}
\frac{1}{\tau_N}
  \sum_{k=0}^{N-1} \alpha_k  
  \dist(-\nabla \bar{f}(\bx_k),\overline{\partial}_{\boldb_k}\cI_{\cX}(\bx_k) )^2 \le
  \frac{(u+2u\ell)}{\tau_N}
\sum_{i=0}^{\chi(N)} \max_{k\in [\cK(i),\cK(i+1)-1]}\boldb_k
  +\frac{\osc(\bar{f})}{\tau_N}.
\end{equation*}
Using again that $\tau_N\ge \frac{1}{2}\chi(N)$, it suffices to show that 
$$
\lim_{N\to\infty}\frac{\sum_{i=0}^{\chi(N)} \max_{k\in [\cK(i),\cK(i+1)-1]}\boldb_k}{\chi(N)}
$$
Propostion~\ref{prop:convergence} implies that there is an integer $i_1$ such that if $i\ge i_1$, then $\max_{k\in [\cK(i),\cK(i+1)-1]}\boldb_k\le h_i(\delta_i)$, where $\delta_i=\frac{\sum_{j=\cK(i)}^{\cK(i+1)-1}\alpha_j^2}{\sum_{k=0}^{\infty}\alpha_k^2}$. 

Furthermore, Proposition~\ref{prop:convergence} implies that $h_i(\delta_i)\to 0$. In particular, given any $\epsilon>0$, there is  number $i_2\ge i_1$ such that if $i\ge i_2$, then $h_i(\delta_i)\le \epsilon$. In particular, for all $i\ge i_2$, we have $\max_{k\in[\cK(i),\cK(i+1)-1]}\boldb_k\le \epsilon. $ So, similar to the expected value case, we have:
$$
\frac{\sum_{i=0}^{\chi(N)} \max_{k\in [\cK(i),\cK(i+1)-1]}\boldb_k}{\chi(N)}\le 
\frac{\sum_{i=0}^{i_2} \max_{k\in [\cK(i),\cK(i+1)-1]}\boldb_k}{\chi(N)}+\epsilon \frac{\chi(N)-i_2}{\chi(N)}.
$$
The first term on the right converges to $0$, while the second is arbitrarily small.

Additionally, Proposition~\ref{prop:convergence} implies that $\boldb_k\to 0$ with probability $1$.  
\end{proofof}

\begin{proofof}{Corollary~\ref{cor:convergence_constant_step}}
  \label{proof:corollary_connstant_step}

  For a constant step size, the construction in Section \ref{sub:preliminaries} reduces to: $s_{i+1} - s_i = \alpha \floor{\frac{1}{\alpha}} \le 1$. Thus, we have $\chi(N) +  1= \ceil{\frac{N}{\floor{1/\alpha}}} < \frac{N}{\floor{1/\alpha}} +1  $. If $\alpha \le \frac{1}{2}$, $\alpha \floor{\frac{1}{\alpha}} > \alpha (\frac{1}{ \alpha}-1) > \frac{1}{2}$, so $\floor{\frac{1}{\alpha}} > \frac{1}{ 2 \alpha}$ and $\chi(N) < 2 \alpha N $. Therefore, the general bound in \eqref{eq:ave_sum} can be simplified as 
  \begin{align*}
    &\frac{1}{\tau_N} \sum_{i=0}^{\chi(N)} \bbE\left[ \int_{s_i}^{s_{i+1}} \left\|\Pi_{T_{\cX}(\bx_t^{C_i})}\left(- \nabla \bar{f}(\bx_t^A)\right) \right\|^2 dt \right] \\
    & \le \frac{1}{\alpha N} \left( \left(\frac{N}{\floor{1/\alpha}} +1\right) (u + 2 u \ell) \left(g_1(1) \sqrt{ \floor{\frac{1}{\alpha}} \alpha^2 }+ g_2(1) \sqrt{\alpha} \right) + osc(\bar{f})\right) \\
    & < 2 (1+ 2 \ell)u  (g_1(1) + g_2(1)) \alpha^{\frac{1}{2}} + \frac{\left(osc(\bar{f}) + (1+ 2 \ell)u  (g_1(1) + g_2(1)) \right)}{\alpha N} 
  \end{align*}
  where functions $g_1$ and $g_2$ were defined in the proof of Theorem~\ref{thm:convergence_general} and the last inequality holds because $\alpha^{-1/2} < \alpha^{-1}$ for any $0 < \alpha <1$.

  To derive the bound on $\boldb_k$, note that $\cK(i+1)-\cK(i)=\floor*{\frac{1}{\alpha}}\le \frac{1}{\alpha}$. Similar to above, we have $\sum_{j=\cK(i)}^{\cK(i+1)-1}\alpha_j^2 \le \alpha$. The bound then follows from Lemma~\ref{lem:x_A2C_b}.

  For the high probability bound, we apply the Theorem~\ref{thm:convergence_general} with $\delta_i = \frac{\delta}{\chi(N)+1}$.
  Then, we have with probability at least $\delta$
  $$
  \frac{1}{N}\sum_{k=0}^{N-1} \dist\left(-\nabla \bar{f}(\bx_k),\overline{\partial}_{h_{\zeta(k)}\left(\frac{\delta}{\chi(N)+1}\right)} \cI_{\cX}(\bx_k) \right)^2 \le 
  \frac{u+2u\ell}{\alpha N}\sum_{i=0}^{\chi(N)}h_i\left(\frac{\delta}{\chi(N)+1}\right) + \frac{Du}{\alpha N}. 
  $$

  Similar to the bound on $\boldb_k$ above, we use that $\sum_{j=\cK(i)}^{\cK(i+1)-1}\alpha_j^2 \le \alpha$.
    So, we can bound:
  \begin{align*}
    h_i(\delta_i)&\le 
    \left( c_{\ref{const_sum_sq_high_prob}} \sqrt{\alpha \log(2\delta_i^{-1})}  + \left( c_{\ref{const_max_high_prob_1}} \log(2\delta_i^{-1}) + c_{\ref{const_max_high_prob_2}}\right)  \alpha \right)^{1/2}
    +c_{\ref{const_general_2}} \sqrt{ \alpha} \\
                 &\le  \underbrace{\left(\left( c_{\ref{const_sum_sq_high_prob}} \sqrt{ \log(2\delta_i^{-1})}  + \left( c_{\ref{const_max_high_prob_1}} \log(2\delta_i^{-1}) + c_{\ref{const_max_high_prob_2}}\right)   \right)^{1/2}
                 +c_{\ref{const_general_2}}\right)}_{q(\delta_i)/(u+2 u\ell)} \alpha^{1/4}.
  \end{align*}
  Using again that $\chi(N)+1\le 2\alpha N+1$ gives:
  \begin{align*}
    \MoveEqLeft
\frac{1}{N}\sum_{k=0}^{N-1} \dist\left(-\nabla \bar{f}(\bx_k),\overline{\partial}_{q\left(\frac{\delta}{2\alpha N+1}\right)\alpha^{1/4}} \cI_{\cX}(\bx_k) \right)^2
\\
&\le 
    \frac{1}{N}\sum_{k=0}^{N-1} \dist\left(-\nabla \bar{f}(\bx_k),\overline{\partial}_{q\left(\frac{\delta}{\chi(N)+1}\right)\alpha^{1/4}} \cI_{\cX}(\bx_k) \right)^2\\
  &\le  2 q\left(\frac{\delta}{\chi(N)+1}\right)\alpha^{1/4}+ \frac{q\left(\frac{\delta}{\chi(N)+1}\right)+Du}{\alpha N} \\
  &\le   2 q\left(\frac{\delta}{2\alpha N+1}\right)\alpha^{1/4}+ \frac{q\left(\frac{\delta}{2\alpha N+1}\right)+Du}{\alpha N} .
               \end{align*}
\end{proofof}

\section{Supporting Lemmas} \label{app:supporting_lemmas}

This sections collects supporting lemmas which bound a series of intermediate processes.

The following lemma is directly used to prove Theorem~\ref{thm:convergence_general}.

\begin{lemma} \label{lem:x_A2C}
  Assume $\bx_{\tau_{k_0}}^A = \bx_{\tau_{k_0}}^C \in \cX$ and $\alpha_k <\frac{1}{2}$ for all $k \in \bbN$, $k \ge k_0$ and $t \in [\tau_k, \tau_{k+1})$, the following bounds hold 
  \begin{enumerate}[label=(\roman*)]
  \item If $\bz_k$ satisfies Assumption \ref{assumption:bounded_var}, then
  \begin{align*}
   &\bbE\left[ \left\|\bx_{t}^A - \bx_{t}^C \right\| \right] \le \sigma e^{\ell(\tau_k-\tau_{k_0} )}  \sqrt{\sum_{j={k_0}}^{k-1} \alpha_j^2}  \\
   & \hspace{100pt}+ e^{\ell (\tau_{k}-\tau_{k_0})} \left( u +  \sqrt{2 r^{-1}  u \left((\tau_{k}-\tau_{k_0}) D u  + D^2 \right)   }  \right) \max_{j \in [k_0,k]} \sqrt{ \alpha_j}. 
  \end{align*}
  \item If $\bz_k$ satisfies Assumption \ref{assumption:L-mixing}, then
  \begin{align*}
    &\bbE\left[\left\|\bx_{t}^A - \bx_{t}^C\right\|\right] \le 2 \ell \Psi_2 (\bz) e^{\ell(\tau_k-\tau_{k_0})}   \sqrt{\sum_{j=k_0}^{k-1} \alpha_j^2}  \\
    & \hspace{100pt} +e^{\ell (\tau_{k}-\tau_{k_0})} \left( u +  \sqrt{2 r^{-1}  u \left((\tau_{k}-\tau_{k_0}) D u  + D^2 \right)   }  \right) \max_{j \in [k_0,k]} \sqrt{ \alpha_j}. 
  \end{align*}
  \item If $\bz_k$ satisfies Assumption \ref{assumption:sub-Gaussian} and $\epsilon >0$, then with probability at least $(1-  e^{-\epsilon})^{2}$, 
  \begin{align*}
    &\sup_{s \in [\tau_{k_0},\tau_k)}\left\|\bx_{s}^A - \bx_{s}^C \right\| \le  \\
    &e^{\ell(\tau_k - \tau_{k_0})} \left( 2 \sqrt{2} \hat{\sigma} D \sqrt{\epsilon}\sqrt{ \sum_{j=k_0}^{k-1} \alpha_j^2}  + \left( 4 \hat{\sigma}^2  \epsilon + \hat{\sigma}^2  + n \hat{\sigma}^2\right)  \sum_{j=k_0}^{k-1} \alpha_j^2 \right)^{1/2}\\
    & \hspace{100pt} +e^{\ell (\tau_{k}-\tau_{k_0})} \left( u +  \sqrt{2 r^{-1}  u \left((\tau_{k}-\tau_{k_0}) D u  + D^2 \right)   }  \right) \max_{j \in [k_0,k]} \sqrt{ \alpha_j}. 
  \end{align*}
  \end{enumerate}
\end{lemma}
%

Before proving Lemma~\ref{lem:x_A2C}, we will show how it can be used to prove Lemma~\ref{lem:x_A2C_b} and Proposition~\ref{prop:convergence} from the main text.

\begin{proofof}{Lemma~\ref{lem:x_A2C_b}}

  Recall that Assumption~\ref{assumption:sub-Gaussian} implies Assumption~\ref{assumption:bounded_var} with $\sigma = \sqrt{n}\hat\sigma$. Let $k_0 = \cK(i)$ and assume that $k\le \cK(i+1)-1$. In this case, $\tau_{k}-\tau_{k_0}\le s_{i+1}-s_i\le 1$. So, we can plug the upper bound of $1$ into all of the $\tau_{k}-\tau_{k_0}$ terms in Lemma~\ref{lem:x_A2C}. Furthermore, $k\le \cK(i+1)-1$ implies that 
  $$
  \sum_{j=k_0}^{k-1}\alpha_j^2 \le \sum_{j=\cK(i)}^{\cK(i+1)-1}\alpha_j^2 \quad \textrm{and}
  \quad \max_{j\in[k_0,k]}\sqrt{\alpha_j}\le \max_{j\in [\cK(i),\cK(i+1))}\sqrt{\alpha_j}.
  $$
  Plugging these bounds into Lemma~\ref{lem:x_A2C} gives the bounds in expectation. 

  To get the bounds in high probability, we do the substitutions above. Furtheremore, note that $(1-e^{-\epsilon})^2\ge 1-2e^{-\epsilon}$. Set $\delta = 2e^{-\epsilon}$, which gives $\epsilon = \log(2\delta^{-1})$. Substituting this value for $\epsilon$ gives result.
\end{proofof}

\begin{proofof}{Proposition~\ref{prop:convergence}}

  By Lemma~\ref{lem:x_A2C_b}, the event $\max_{k\in [\cK(i),\cK(i+1))}\boldb_k > h_i(\delta_i)$ occurs with probability at most $\delta_i$. By construction, 
  $$
  \sum_{i=0}^{\infty}\delta_i = 1,
  $$
So, the Borel-Cantelli Lemma implies that $\max_{k\in [\cK(i),\cK(i+1))}\boldb_k > h_i(\delta_i)$ can occur at most finitely many times. 

To complete the proof, it suffices to show that when $i \rightarrow \infty$, $h_i(\delta_i)\to 0$. Note that $\alpha_k\to 0$ and $\cK(i)\to \infty$. Thus, 
$$
\lim_{i\to\infty}\max_{k\in[\cK(i),\cK(i+1))}\sqrt{\alpha_k} = 0. 
$$
Similarly, $\sum_{k=0}^{\infty}\alpha_k^2<\infty$ and $\cK(i) \rightarrow \infty$ implies that
$$
\lim_{i\to\infty}\sum_{j=\cK(i)}^{\cK(i+1)-1}\alpha_j^2 = 0.
$$

Thus, to show that $h_i(\delta_i)=0$, using $\delta_i = \frac{\sum_{j=\cK(i)}^{\cK(i+1)-1}\alpha_j^2}{\sum_{k=0}^{\infty}\alpha_k^2}$, it suffices to show that
$$
\log\left(\frac{1}{\sum_{j=\cK(i)}^{\cK(i+1)-1}\alpha_j^2} \right) \sum_{j=\cK(i)}^{\cK(i+1)-1}\alpha_j^2 \to 0.
$$
This is now a special case of $\lim_{t\downarrow 0}t\log(t^{-1})=0.$ 
\end{proofof}

The following lemmas support the proof of Lemma~\ref{lem:x_A2C}. 
 
\begin{lemma} \label{lem:x_A2M_bounded_var}
  Assume $\bx_{\tau_{k_0}}^A = \bx_{\tau_{k_0}}^M \in \cX$ and $\bz_k$ satisfies assumption \ref{assumption:bounded_var}, for all $k \in \bbN$, $k \ge k_0$, the following bound holds:
  $$
  \bbE\left[\|\bx_{\tau_k}^A - \bx_{\tau_k}^M\|\right] \le \sigma e^{\ell( \tau_k -\tau_{k_0})} \sqrt{\sum_{s=k_0}^{k-1} \alpha_s^2}. 
  $$
\end{lemma}

\begin{lemma} \label{lem:x_A2M_L_mixing}
  Assume $\bx_{\tau_{k_0}}^A = \bx_{\tau_{k_0}}^M \in \cX$ and $\bz_k$ satisfies assumption \ref{assumption:L-mixing}, for all $k \in \bbN$, $k \ge k_0$, the following bound holds:
  $$
  \bbE\left[\left\|\bx_{\tau_k}^A - \bx_{\tau_k}^M\right\|\right] \le   2 \ell \Psi_2 (\bz) e^{\ell( \tau_k -\tau_{k_0})} \sqrt{\sum_{s=k_0}^{k-1} \alpha_s^2}.
  $$
\end{lemma}

\begin{lemma} \label{lem:x_A2M_high_P}
  Assume $\bx_{\tau_{k_0}}^A = \bx_{\tau_{k_0}}^M \in \cX$, $\bz_k$ satisfies assumption \ref{assumption:sub-Gaussian}, $\alpha_k \le \frac{1}{2}$, for all $k \in \bbN$, $k \ge k_0$ and $\epsilon>0$, then with probability at least $(1 -  e^{-\epsilon})^2 $, the following bound holds:
  \begin{align*}
    \max_{s \in [k_0,k]}
    \left\|\bx_{\tau_s}^A - \bx_{\tau_s}^M \right\| \le e^{\ell(\tau_k - \tau_{k_0})} \left( 2 \sqrt{2} \hat{\sigma} D \sqrt{\epsilon}\sqrt{ \sum_{j=k_0}^{k-1} \alpha_j^2}  + \left( 4 \hat{\sigma}^2  \epsilon + \hat{\sigma}^2  + n \hat{\sigma}^2\right)  \sum_{j=k_0}^{k-1} \alpha_j^2 \right)^{1/2}.
  \end{align*}
\end{lemma}



%
\begin{lemma} \label{lem:x_C2D}
  Assume that $\bx_{\tau_{k_0}}^C = \bx_{\tau_{k_0}}^D$, for all $k \in \bbN$, $k \ge k_0$, the following bound holds
  \begin{align*}
    \left\|\bx_{\tau_k}^C - \bx_{\tau_k}^D \right\| \le \sqrt{2 r^{-1}  u \left((\tau_k -\tau_{k_0}) D u  + D^2\right) \max_{j \in [k_0,k]} \alpha_j}.
  \end{align*}
\end{lemma}

\begin{lemma} \label{lem:x_C2C}
  For all $t \in [ \tau_k, \tau_{k+1})$, the following bound holds
  $$
  \left\|\bx_t^C - \bx_{\tau_k}^C \right\| \le \alpha_k u.
  $$
\end{lemma}

\begin{lemma} \label{lem:x_C2D_continuous}
  Assume $\bx_{\tau_{k_0}}^C = \bx_{\tau_{k_0}}^D \in \cX$, for all $t \in [\tau_k,\tau_{k+1})$ where $k \in \bbN$, $k \ge k_0$, the following bound holds
  \begin{align*}
    \left\|\bx_{t}^C - \bx_t^D \right\| \le \alpha_k u + \sqrt{2 r^{-1} u \left((\tau_k -\tau_{k_0}) D u  + D^2\right) \max_{j \in [k_0,k]} \alpha_j}.
  \end{align*}
\end{lemma}


\begin{lemma} \label{lem:x_M2D}
  Assume $\bx_{\tau_{k_0}}^M = \bx_{\tau_{k_0}}^D \in \cX$, $\alpha_k \le \frac{1}{2}$ for all $k \in \bbN$, $k \ge k_0$, the following bound holds 
  \begin{align*}
     \left\|\bx_{\tau_k}^M - \bx_{\tau_k}^D\right\| \le (e^{\ell (\tau_{k}-\tau_{k_0})} -1) \max_{s \in [k_0, k]} \sqrt{\alpha_s}\left( u+ \sqrt{2 r^{-1}  u \left((\tau_{k}-\tau_{k_0}) D u  + D^2 \right)  }\right).
  \end{align*}
\end{lemma}

\begin{proofof}{Lemma~\ref{lem:x_A2C}}
  
  For $t \in [\tau_{k}, \tau_{k+1})$, $\bx_t^A = \bx_{\tau_k}^A$, then the triangle inequality gives
\begin{align*}
    \|\bx_{t}^A - \bx_{t}^C\| \le \|\bx_{\tau_k}^A - \bx_{\tau_k}^M\| + \|\bx_{\tau_k}^M- \bx_{\tau_k}^D\| + \|\bx_{\tau_k}^D - \bx_{\tau_k}^C\| +\|\bx_{t}^C - \bx_{\tau_k}^C\|.
\end{align*}

For part (i), under Assumption \ref{assumption:bounded_var}, combininig Lemma~\ref{lem:x_A2M_bounded_var}, Lemma~\ref{lem:x_M2D}, Lemma~\ref{lem:x_C2D} and Lemma~\ref{lem:x_C2C} gives  

\begin{align*} 
  & \bbE\left[\left\|\bx_{t}^A - \bx_{t}^C \right\| \right]\\
  &\le \sigma e^{\ell(-\tau_{k_0} + \tau_k)} \sqrt{\sum_{j=k_0}^{k-1} \alpha_j^2} +  (e^{\ell (\tau_{k}-\tau_{k_0})} -1) \max_{j \in [k_0, k]} \sqrt{\alpha_j}\left( u+ \sqrt{2 r^{-1}  u \left((\tau_{k}-\tau_{k_0}) D u  + D^2 \right)  }\right)  \nonumber\\
  & \hspace{190pt}+ \sqrt{2 r^{-1}  u \left((\tau_{k}-\tau_{k_0}) D u  + D^2 \right)  \max_{j \in [k_0,k]} \alpha_j  } + \alpha_k u\\
  & \le \sigma e^{\ell(\tau_k -\tau_{k_0})}  \sqrt{\sum_{j=0}^{k-1} \alpha_j^2}  +e^{\ell (\tau_{k} -\tau_{k_0})} \left( u +  \sqrt{2 r^{-1}  u \left((\tau_{k}- \tau_{k_0}) D u  + D^2 \right)   }  \right) \max_{j \in [k_0,k]} \sqrt{ \alpha_j}  \nonumber
\end{align*}
where the last inequality uses that $\alpha_k \le \sqrt{\alpha_k}$ for all $\alpha_k \le \frac{1}{2}$.

For part (ii), under Assumption \ref{assumption:L-mixing}, combining Lemma~\ref{lem:x_A2M_L_mixing}, Lemma~\ref{lem:x_M2D}, Lemma~\ref{lem:x_C2D} and Lemma~\ref{lem:x_C2C}  gives the desired result.

For part (iii), under Assumption \ref{assumption:sub-Gaussian}, combining Lemma~\ref{lem:x_A2M_high_P}, Lemma~\ref{lem:x_M2D}, Lemma~\ref{lem:x_C2D} and Lemma~\ref{lem:x_C2C}  gives the desired result.
\end{proofof}

\begin{proofof}{Lemma~\ref{lem:x_A2M_bounded_var}}

  We introduce another intermediate process where $\bx_{\tau_{k_0}}^B = \bx_{\tau_{k_0}}^M$:
  \begin{align} \label{eq:x_B}
    \bx_{\tau_{k+1}}^B = \Pi_{\cX}\left( \bx_{\tau_k}^B  - \alpha_k \nabla f(\bx_{\tau_k}^M, \bz_k)\right).
  \end{align}
  
 The  triangle inequality gives
  \begin{align} \label{eq:x_B_decomposed}
    \left\|\bx_{\tau_k}^A - \bx_{\tau_k}^M \right\| \le \left\|\bx_{\tau_k}^A - \bx_{\tau_k}^B \right\| + \left\|\bx_{\tau_k}^B - \bx_{\tau_k}^M\right\|.
  \end{align}
  
  Bound the first term on the RHS of \eqref{eq:x_B_decomposed} as:
  \begin{equation} \label{eq:x_A2B}
  \begin{aligned}
      \left\|\bx_{\tau_{k+1}}^A - \bx_{\tau_{k+1}}^B\right\| & \le \left\|\bx_{\tau_{k}}^A - \alpha_k \nabla f(\bx_{\tau_k}^A, \bz_k) - \left( \bx_{\tau_{k}}^B - \alpha_k \nabla f(\bx_{\tau_k}^M, \bz_k)\right) \right\| \\
      & \le \left\|\bx_{\tau_{k}}^A - \bx_{\tau_{k}}^B \right\| +\alpha_k \left\|\nabla f(\bx_{\tau_k}^A, \bz_k) - \nabla f(\bx_{\tau_k}^M, \bz_k) \right\| \\
      & \le \left\|\bx_{\tau_{k}}^A - \bx_{\tau_{k}}^B\| +\alpha_k \ell \| \bx_{\tau_k}^A - \bx_{\tau_k}^M\right\| \\
      & \le (1 + \alpha_k \ell )\left\|\bx_{\tau_{k}}^A - \bx_{\tau_{k}}^B\| +\alpha_k \ell \|\bx_{\tau_{k}}^B - \bx_{\tau_{k}}^M\right\|.
  \end{aligned}
  \end{equation}
  
  Bound the second term on the RHS of \eqref{eq:x_B_decomposed} as:
  \begin{equation} \label{eq:x_B2M}
  \begin{aligned}
    \left\|\bx_{\tau_{k+1}}^B - \bx_{\tau_{k+1}}^M \right\|^2 & \le \left\| \bx_{\tau_{k}}^B - \alpha_k \nabla f(\bx_{\tau_k}^M, \bz_k) - \left( \bx_{\tau_{k}}^M - \alpha_k \nabla \bar{f}(\bx_{\tau_k}^M)\right) \right\|^2 \\
    & = \left\|\bx_{\tau_{k}}^B - \bx_{\tau_{k}}^M\|^2 + \alpha_k^2 \|\nabla f(\bx_{\tau_k}^M, \bz_k) - \nabla \bar{f}(\bx_{\tau_k}^M )\right\|^2 \\
    & - 2 \alpha_k \left(\bx_{\tau_{k}}^B - \bx_{\tau_{k}}^M\right)^\top \left(\nabla f(\bx_{\tau_k}^M, \bz_k) - \nabla \bar{f}(\bx_{\tau_k}^M )\right).
  \end{aligned}
  \end{equation}

  Since $\bx_0$ is independent of all $\bz_k$ and all $\bz_k$ are independent, taking the expectation of the cross term of \eqref{eq:x_B2M} gives
  \begin{align*}
    &\bbE\left[ \left(\bx_{\tau_{k}}^B - \bx_{\tau_{k}}^M\right)^\top \left(\nabla f(\bx_{\tau_k}^M, \bz_k) - \nabla \bar{f}(\bx_{\tau_k}^M )\right) \right] \\
    =& \bbE\left[ \bx_{\tau_{k}}^B - \bx_{\tau_{k}}^M\right]^\top \bbE\left[\nabla f(\bx_{\tau_k}^M, \bz_k) - \nabla \bar{f}(\bx_{\tau_k}^M )\right]  
    = 0.
  \end{align*}
  
  Therefore, taking expectation over \eqref{eq:x_B2M} gives
  \begin{align} 
    \bbE\left[\left\|\bx_{\tau_{k+1}}^B - \bx_{\tau_{k+1}}^M\right\|^2 \right] & \le \bbE\left[\left\|\bx_{\tau_{k}}^B - \bx_{\tau_{k}}^M\right\|^2 \right] + \alpha_k^2 \sigma^2.
  \end{align}
  Iterating and Jensen's inequality gives
  \begin{align} \label{eq:exp_x_B2M}
    \bbE\left[\left\|\bx_{\tau_{k}}^B - \bx_{\tau_{k}}^M\right\| \right]\le  \sigma \sqrt{ \sum_{j=k_0}^{k-1}\alpha_j^2 }.
  \end{align}
  
  Taking expectation over \eqref{eq:x_A2B} and plugging \eqref{eq:exp_x_B2M}, we get 
  \begin{align} \label{eq:exp_x_A2B}
    \bbE[\|\bx_{\tau_{k}}^A - \bx_{\tau_{k}}^B\|] &\le (1 + \alpha_{k-1} \ell) \bbE[\|\bx_{\tau_{k-1}}^A - \bx_{\tau_{k-1}}^B\|] + \alpha_{k-1} \ell \sqrt{\sum_{j=k_0}^{k-2} \alpha_j^2} \sigma  \nonumber \\
    & \le \sum_{i=k_0+1}^{k-1} \Pi_{j = i+1}^{k-1}(1 + \alpha_{j} \ell) \alpha_i \ell \sqrt{\sum_{s= k_0}^{i-1} \alpha_s^2} \sigma  \nonumber\\
    & \le \sum_{i=k_0+1}^{k-1} e^{\ell (\tau_k - \tau_{i+1})} \alpha_{i} \ell \sqrt{\sum_{s=k_0}^{i-1} \alpha_s^2} \sigma \nonumber\\
    & \le e^{\ell \tau_k}  \ell \int_{\tau_{k_0+1}}^{\tau_{k}} e^{-\ell w} dw \sqrt{\sum_{s={k_0}}^{k-2} \alpha_s^2} \sigma \nonumber \\
    & \le (e^{\ell(-\tau_{k_0+1} + \tau_k)}-1) \sqrt{\sum_{s=k_0}^{k-2} \alpha_s^2} \sigma \nonumber \\
    & \le (e^{\ell( \tau_k-\tau_{k_0})}-1) \sqrt{\sum_{s=k_0}^{k-1} \alpha_s^2} \sigma
  \end{align}
  where the third inequality is because $1 + x \le e^x$ for all $x \ge 0$ and the second to the last inequality uses a Riemann sum bound.

  Combining \eqref{eq:exp_x_B2M} and \eqref{eq:exp_x_A2B} completes the proof.
  
  \end{proofof}

  \begin{proofof}{Lemma~\ref{lem:x_A2M_L_mixing}}

    To obtain the desired bound, we further introduce the following two intermediate processes:
    \begin{subequations}
    \begin{equation}
      \bx_{\tau_{k+1}}^{M,s} = \Pi_{\cX} \left(\bx_{\tau_{k}}^{M,s} - \alpha_k \bbE\left[\nabla f(\bx_{\tau_k}^{M,s}, \bz_k) \vert \cF_{k-s} \vee \cG\right]\right)
    \end{equation}
    \begin{equation}
      \bx_{\tau_{k+1}}^{B,s} = \Pi_{\cX} \left(\bx_{\tau_{k}}^{B,s} - \alpha_k \bbE\left[\nabla f(\bx_{\tau_k}^{M,s}, \bz_k) \vert \cF_{k-s-1} \vee \cG \right]\right)
    \end{equation}
  \end{subequations}
    where $\cG = \sigma(\{\bx_0\})$. We set $\cF_j =  \{\emptyset, \cZ\}$ for all $j < 0$. $\bx_{\tau_{k_0}}^{M,s} = \bx_{\tau_{k_0}}^{B,s} = \bx_{\tau_{k_0}}^A$ for all $s \ge 0$. For $s=0$, $\bx_{\tau_k}^{M,0} = \bx_{\tau_k}^{A}$ and for $ s> k$, $\bx_{\tau_k}^{M,s} = \bx_{\tau_k}^M$.
  
    Therefore, using the triangle inequality, we have 
    \begin{equation} \label{eq:x_A2M_decomposed}
    \begin{aligned}[b]
      \left\|\bx_{\tau_k}^A - \bx_{\tau_k}^M\right\| & \le \sum_{s= 0}^{k} \left\|\bx_{\tau_k}^{M,s} - \bx_{\tau_k}^{M,s+1} \right\| 
      & \le \sum_{s= 0}^{k} \left\|\bx_{\tau_k}^{M,s} - \bx_{\tau_k}^{B,s} \right\| + \sum_{s= 0}^{k} \left\|\bx_{\tau_k}^{B,s} - \bx_{\tau_k}^{M,s+1}\right\|.
    \end{aligned}
    \end{equation}
    In the following, we want to bound $\bbE\left[\left\|\bx_{\tau_k}^{M,s} - \bx_{\tau_k}^{B,s}\right\|\right]$ and $\bbE\left[\left\|\bx_{\tau_k}^{B,s} - \bx_{\tau_k}^{M,s+1}\right\|\right]$.
    \begin{align} \label{eq:x_M2Bs}
      & \left\|\bx_{\tau_{k+1}}^{M,s} - \bx_{\tau_{k+1}}^{B,s}\right\|^2 \nonumber \\
      & \le \left\|\bx_{\tau_{k}}^{M,s} - \bx_{\tau_{k}}^{B,s} -\alpha_k \left( \bbE[\nabla f(\bx_{\tau_k}^{M,s}, \bz_k) \vert \cF_{k-s} \vee \cG] - \bbE[\nabla f(\bx_{\tau_k}^{M,s}, \bz_k) \vert \cF_{k-s-1} \vee \cG ]\right) \right\|^2 \nonumber \\
      & = \left\|\bx_{\tau_{k}}^{M,s} - \bx_{\tau_{k}}^{B,s} \right\|^2 + \alpha_k^2 \left\|\bbE[\nabla f(\bx_{\tau_k}^{M,s}, \bz_k) \vert \cF_{k-s} \vee \cG] - \bbE\left[\nabla f(\bx_{\tau_k}^{M,s}, \bz_k) \vert \cF_{k-s-1} \vee \cG \right] \right\|^2  \\
      & - 2 \alpha_k \left(\bx_{\tau_{k}}^{M,s} - \bx_{\tau_{k}}^{B,s} \right)^\top  \left( \bbE\left[\nabla f(\bx_{\tau_k}^{M,s}, \bz_k) \vert \cF_{k-s} \vee \cG \right] - \bbE\left[\nabla f(\bx_{\tau_k}^{M,s}, \bz_k) \vert \cF_{k-s-1} \vee \cG\right]\right) \nonumber
    \end{align}
  
    We can show that the cross term has zero mean. By definition, $\bx_{\tau_k}^{M,s}$ is $\cF_{k-s-1}\vee \cG $-measurable and $\bx_{\tau_k}^{B,s}$ is $\cF_{k-s-2}\vee \cG $-measurable. Therefore, we have the following 
    \begin{align*}
      &\bbE\left[ \left(\bx_{\tau_{k}}^{M,s} - \bx_{\tau_{k}}^{B,s} \right)^\top  \left( \bbE\left[\nabla f(\bx_{\tau_k}^{M,s}, \bz_k) \vert \cF_{k-s} \vee \cG \right] - \bbE\left[\nabla f(\bx_{\tau_k}^{M,s}, \bz_k) \vert \cF_{k-s-1} \vee \cG \right]\right) \right] \\
      & = \bbE\left[ \left(\bx_{\tau_{k}}^{M,s} - \bx_{\tau_{k}}^{B,s} \right)^\top  \left( \bbE\left[ \bbE\left[\nabla f(\bx_{\tau_k}^{M,s}, \bz_k) \vert \cF_{k-s} \vee \cG \right] \Big| \cF_{k-s-1} \vee \cG\right] 
      \right. \right. \\
  &\left. \left. \hspace{160pt}
      - \bbE\left[ \bbE\left[\nabla f(\bx_{\tau_k}^{M,s}, \bz_k) \vert \cF_{k-s-1} \vee \cG \right] \Big| \cF_{k-s-1} \vee \cG \right]\right) \right] \\
      &=0.
    \end{align*}
  
    For the second term of \eqref{eq:x_M2Bs}, 
    \begin{align*}
      & \left\|\bbE\left[\nabla f(\bx_{\tau_k}^{M,s}, \bz_k) \vert \cF_{k-s} \vee \cG  \right] - \bbE\left[\nabla f(\bx_{\tau_k}^{M,s}, \bz_k) \vert \cF_{k-s-1} \vee \cG  \right] \right\|^2 \\
      & \le 2 \left\|\bbE\left[\nabla f(\bx_{\tau_k}^{M,s}, \bz_k) \vert \cF_{k-s} \vee \cG  \right] - \bbE\left[\nabla f\left(\bx_{\tau_k}^{M,s}, \bbE\left[\bz_k \vert \cF_{k-s}^+ \right]\right) \vert \cF_{k-s} \vee \cG  \right] \right\|^2  \\
      & +  2 \left\|\bbE[\nabla f(\bx_{\tau_k}^{M,s}, \bz_k) \vert \cF_{k-s-1} \vee \cG ] - \bbE[\nabla f\left(\bx_{\tau_k}^{M,s}, \bbE[\bz_k \vert \cF_{k-s}^+] \right) \vert \cF_{k-s}\vee \cG  ] \right\|^2  \\
      & \le 2 \ell^2 \bbE\left[\|\bz_k -   \bbE[\bz_k \vert \cF_{k-s}^+]\|^2 \vert \cF_{k-s} \vee \cG \right] + 2 \ell^2 \bbE\left[\|\bz_k -   \bbE[\bz_k \vert \cF_{k-s}^+ ]\|^2 \vert \cF_{k-s-1} \vee \cG \right].
    \end{align*} 

    Taking expectation and plugging in the $L$-mixing property gives
    \begin{align} \label{eq:L_mixing_ineq}
      \bbE\left[ \left\|\bbE[\nabla f(\bx_{\tau_k}^{M,s}, \bz_k) \vert \cF_{k-s} \vee \cG ] - \bbE[\nabla f(\bx_{\tau_k}^{M,s}, \bz_k) \vert \cF_{k-s-1} \vee \cG ] \right\|^2 \right] \le 4 \ell^2 \psi_2(s,\bz)^2.
    \end{align}
  
    Therefore, taking expectation of \eqref{eq:x_M2Bs} and plugging in \eqref{eq:L_mixing_ineq}, we have 
    \begin{align*}
      & \bbE\left[ \left\|\bx_{\tau_{k+1}}^{M,s} - \bx_{\tau_{k+1}}^{B,s} \right\|^2 \right] \le \bbE\left[ \left\|\bx_{\tau_{k}}^{M,s} - \bx_{\tau_{k}}^{B,s} \right\|^2 \right] + 4 \ell^2 \psi_2(s,\bz)^2 \alpha_k^2.
    \end{align*}
    
    Iterating gives
    \begin{align*}
      \bbE\left[ \left\|\bx_{\tau_{k}}^{M,s} - \bx_{\tau_{k}}^{B,s} \right\|^2 \right] \le  4 \ell^2 \psi_2(s,\bz)^2\sum_{j=k_0}^{k-1} \alpha_{j}^2.
    \end{align*}
  
    Jensen's inequality gives
    \begin{align} \label{eq:x_M2B}
      \bbE\left[ \left\|\bx_{\tau_{k}}^{M,s} - \bx_{\tau_{k}}^{B,s}\right\| \right] \le 2 \ell \psi_2(s,\bz) \sqrt{\sum_{j=k_0}^{k-1} \alpha_{j}^2}.
    \end{align}
  
    Now, we proceed to bound $\bbE\left[ \left\|\bx_{\tau_k}^{B,s} - \bx_{\tau_k}^{M,s+1} \right\|\right]$.
    \begin{align*}
      & \left\|\bx_{\tau_{k+1}}^{B,s} - \bx_{\tau_{k+1}}^{M,s+1} \right\| \\
      & \le \left\|\bx_{\tau_{k}}^{B,s} - \bx_{\tau_{k}}^{M,s+1} - \alpha_k \left(  \bbE\left[\nabla f(\bx_{\tau_k}^{M,s}, \bz_k) \vert \cF_{k-s-1} \vee \cG \right] - \bbE\left[\nabla f(\bx_{\tau_k}^{M,s+1}, \bz_k) \vert \cF_{k-s-1} \vee \cG  \right] \right) \right\| \\
      & \le  \left\|\bx_{\tau_{k}}^{B,s} - \bx_{\tau_{k}}^{M,s+1} \right\| + \alpha_k \ell \bbE\left[ \left\|\bx_{\tau_k}^{M,s} - \bx_{\tau_k}^{M,s+1} \right\| \Big| \cF_{k-s-1} \vee \cG  \right]
    \end{align*}
  
   Taking expectation gives 
    \begin{align*}
     \bbE\left[\left\|\bx_{\tau_{k+1}}^{B,s} - \bx_{\tau_{k+1}}^{M,s+1} \right\| \right]  &\le \bbE\left[\left\|\bx_{\tau_{k}}^{B,s} - \bx_{\tau_{k}}^{M,s+1} \right\| \right] + \alpha_k \ell \bbE\left[\left\|\bx_{\tau_k}^{M,s} - \bx_{\tau_k}^{M,s+1} \right\| \right]  \\
      & \le \bbE\left[\left\|\bx_{\tau_{k}}^{B,s} - \bx_{\tau_{k}}^{M,s+1} \right\| \right] + \alpha_k \ell \bbE\left[\left\|\bx_{\tau_k}^{M,s} - \bx_{\tau_k}^{B,s} \right\| \right] + \alpha_k \ell \bbE\left[\left\|\bx_{\tau_k}^{B,s} - \bx_{\tau_k}^{M,s+1} \right\| \right] \\
      & \le (1 + \alpha_k \ell)  \bbE\left[\left\|\bx_{\tau_{k}}^{B,s} - \bx_{\tau_{k}}^{M,s+1} \right\| \right]  + \alpha_k \ell \bbE\left[\left\|\bx_{\tau_k}^{M,s} - \bx_{\tau_k}^{B,s} \right\| \right].
    \end{align*}
  
    Plugging \eqref{eq:x_M2B} and iterating gives
    \begin{equation} \label{eq:x_B2Mplus}
    \begin{aligned}[b]
      \bbE\left[\left\|\bx_{\tau_{k}}^{B,s} - \bx_{\tau_{k}}^{M,s+1} \right\| \right] 
      & \le \sum_{i=k_0}^{k-1} \Pi_{j = i+1}^{k-1} (1+ \alpha_j \ell) \alpha_i \ell \bbE\left[\left\|\bx_{\tau_i}^{M,s} - \bx_{\tau_i}^{B,s} \right\| \right] \\
      & \le \sum_{i =k_0}^{k-1} e^{\ell \sum_{j = i+1}^{k-1} \alpha_j } \alpha_i \ell 2 \ell \psi_2(s,\bz) \sqrt{\sum_{s=k_0}^{i-1} \alpha_s^2} \\
      & \le 2 \ell \psi_2(s,\bz) \sqrt{\sum_{s=k_0}^{k-2} \alpha_s^2} \sum_{i=k_0}^{k-1} e^{\tau_k} e^{-\tau_{i+1}} \alpha_i  \ell \\
      & \le 2 \ell \psi_2(s,\bz) \sqrt{\sum_{s=k_0}^{k-2} \alpha_s^2}  e^{\tau_k} \ell \int_{\tau_{k_0+1}}^{\tau_k} e^{-\ell w} dw  \\
      & \le  2 \ell \psi_2(s,\bz) \sqrt{\sum_{s=k_0}^{k-2} \alpha_s^2}  ( e^{\ell(\tau_k - \tau_{k_0+1})}  -1) \\
      & \le 2 \ell \psi_2(s,\bz)  ( e^{\ell(\tau_k - \tau_{k_0})}  -1) \sqrt{\sum_{s=k_0}^{k-1} \alpha_s^2} .
    \end{aligned}
    \end{equation}
  
  Plugging the bounds from \eqref{eq:x_M2B} and \eqref{eq:x_B2Mplus} into \eqref{eq:x_A2M_decomposed} gives the desired bound.
  \end{proofof}

\begin{proofof}{Lemma~\ref{lem:x_A2M_high_P}}

  $\bx_{\tau_k}^B$ is defined in Lemma~ \ref{lem:x_A2M_bounded_var}. Recall
  \begin{align*}
    \bx_{\tau_{k+1}}^B = \Pi_{\cX}\left( \bx_{\tau_k}^B  - \alpha_k \nabla f(\bx_{\tau_k}^M, \bz_k)\right).
  \end{align*}
  
  Triangle inequality gives
  \begin{align} \label{eq:decomposed_x_A2M}
    \left\| \bx_{\tau_k}^A - \bx_{\tau_k}^M \right\| \le \left\| \bx_{\tau_k}^A - \bx_{\tau_k}^B \right\|  + \left\| \bx_{\tau_k}^B - \bx_{\tau_k}^M \right\|.
  \end{align}
 So the goal is to bound $\| \bx_{\tau_k}^A - \bx_{\tau_k}^B\| $ and $\| \bx_{\tau_k}^B - \bx_{\tau_k}^M\|$.

  Similar to \eqref{eq:exp_x_A2B} but without taking expectaion, 
  we have 
  \begin{equation}
  \begin{aligned} \label{eq:deterministic_x_A2B}
   \left\| \bx_{\tau_k}^A - \bx_{\tau_k}^B \right\| &\le  \sum_{i=k_0+1}^{k-1} \Pi_{j = i+1}^{k-1}(1 + \alpha_{j} \ell) \alpha_i \ell \max_{i \in [k_0, k-1]} \left\|\bx_{\tau_i}^{B} - \bx_{\tau_i}^M \right\|  \\
   & \le  (e^{\ell(\tau_k-\tau_{k_0} )}-1) \max_{i \in [k_0, k-1]} \left\|\bx_{\tau_i}^{B} - \bx_{\tau_i}^M \right\|.
  \end{aligned}
\end{equation}

Thus, we want to bound $\|\bx_{\tau_i}^B - \bx_{\tau_i}^M \|$ for all $i \in [k_0, k-1]$.
  
  Iterating \eqref{eq:x_B2M} gives
  \begin{equation} 
    \label{eq:determinisitc_x_B2M}
    \begin{aligned}
      \|\bx_{\tau_{k}}^B - \bx_{\tau_{k}}^M\|^2 & 
      & \le \sum_{i=k_0}^{k-1} 2 \alpha_i (\bx_{\tau_{i}}^M - \bx_{\tau_{i}}^B)^\top \bz_i  + \sum_{i=k_0}^{k-1}\alpha_i^2 \|\bz_i\|^2 .
    \end{aligned}
    \end{equation}

    In the following, we show how to bound the two terms on the RHS respectively.

    Let $\bv_i=  \bx_{\tau_i}^M - \bx_{\tau_i}^B$ and we have $\|\bv_i\| \le D$ for all $i$ from the assumption on $\cX$. First, we want to show $\max_{s \in [k_0,k-1]} 2 \sum_{i=k_0}^{s} \alpha_i \bv_i^\top \bz_i$ is sub-Gaussian. 
From the uniform sub-Gaussian Assumption \ref{assumption:sub-Gaussian}, we can obtain that for all $\lambda \in \bbR$:
    \begin{align*}
       \bbE\left[e^{\lambda 2 \sum_{i=k_0}^{k-1} \alpha_i \bv_i^\top \bz_i}\right]
      & \le e^{\frac{1}{2}{\lambda^2 4  D^2 \hat{\sigma}^2 } \sum_{i=k_0}^{k-1} \alpha_i^2 }.
  \end{align*}

By definition, $\bv_i^\top \bz_i$ is $\cF_{i} \vee \cG $-measurable, where $\cG$ is defined in Lemma~\ref{lem:x_A2M_L_mixing}. Then, 
  \begin{align*}
    \bbE\left[e^{\lambda 2 \sum_{i=k_0}^{s} \alpha_i \bv_i^\top \bz_i} \Big| \cF_{s-1} \vee \cG \right]
   & \le e^{ \lambda 2 \sum_{i=k_0}^{s-1} \alpha_i \bv_i^\top \bz_i + \frac{1}{2}{\lambda^2 4  D^2 \hat{\sigma}^2 }  \alpha_{s}^2 }.
\end{align*}

Let $M_{s}(\lambda) = e^{  \sum_{i=k_0}^{s} \left( 2 \lambda \alpha_i \bv_i^\top \bz_i - \frac{1}{2}{\lambda^2 4  D^2 \hat{\sigma}^2 }  \alpha_{i}^2 \right)} $. We can show that $M_{s}(\lambda)$ is supermartingale:
\begin{equation}\label{eq:supermartingale}
\begin{aligned}[b] 
  \bbE\left[ M_{s}(\lambda) \vert \cF_{s-1} \vee \cG \right] &\le   e^{  \sum_{i=k_0}^{s-1} \left( 2 \lambda \alpha_i \bv_i^\top \bz_i - \frac{1}{2}{\lambda^2 4  D^2 \hat{\sigma}^2 }  \alpha_{i}^2 \right)} \bbE\left[ e^{2 \lambda \alpha_s \bv_s^\top \bz_s - \frac{1}{2} \lambda^2 4 D^2 \hat{\sigma}^2 \alpha_s^2} \Big| \cF_{s-1} \vee \cG \right]  \\
  &\le M_{s-1}(\lambda).
\end{aligned}
\end{equation}
almost surely for all $s \ge k_0+1$.

By iterating \eqref{eq:supermartingale}, we have for all $s \in [k_0, k-1]$,
\begin{align*}
  \bbE\left[ M_{s}(\lambda) \right] \le 1.
\end{align*}

Using Doob's maximal inequality \citep[see][Theorem 3.9]{lattimore2020bandit} and choosing an $\epsilon >0$, we have 
\begin{align} 
  &\bbP\left( \max_{s \in [k_0,k-1]} M_{s}(\lambda) \ge e^{\epsilon}\right) \le e^{- \epsilon} \bbE\left[M_{k_0} (\lambda)\right]  \le e^{- \epsilon} \nonumber \\
   \Leftrightarrow & \; \bbP\left( \max_{s \in [k_0,k-1]}   \sum_{i=k_0}^{s} \left( 2 \lambda \alpha_i \bv_i^\top \bz_i - \frac{1}{2}{\lambda^2 4  D^2 \hat{\sigma}^2 }  \alpha_{i}^2 \right)  \ge \epsilon\right) \le e^{- \epsilon} \nonumber\\
    \Rightarrow & \;
   \bbP\left( \max_{s \in [k_0,k-1]}   \sum_{i=k_0}^{s} 2 \lambda \alpha_i \bv_i^\top \bz_i   \ge \epsilon + \sum_{i=k_0}^{k-1} \frac{1}{2}{\lambda^2 4  D^2 \hat{\sigma}^2 }  \alpha_{i}^2 \right) \le e^{-\epsilon}  \nonumber\\
   \Leftrightarrow & \;   \bbP\left( \max_{s \in [k_0,k-1]}   \sum_{i=k_0}^{s} 2  \alpha_i \bv_i^\top \bz_i   \ge \frac{\epsilon}{\lambda} + \sum_{i=k_0}^{k-1} \frac{1}{2}{\lambda 4  D^2 \hat{\sigma}^2 }  \alpha_{i}^2 \right) \le e^{-\epsilon}.  \label{eq:lastline}
\end{align}

The RHS of the inequality inside the probability of \eqref{eq:lastline} is minimized at $\lambda^* = \sqrt{\frac{\epsilon}{ \frac{1}{2}{ 4  D^2 \hat{\sigma}^2 }  \sum_{i=k_0}^{k-1}\alpha_{i}^2}}$.

Then plugging $\lambda^*$ into \eqref{eq:lastline} gives 
\begin{align} \label{eq:pr_event_A}
  \bbP\left( \max_{s \in [k_0,k-1]}   \sum_{i=k_0}^{s} 2  \alpha_i \bv_i^\top \bz_i   \ge  2 \sqrt{2} D \hat{\sigma} \sqrt{\epsilon} \sqrt{\sum_{i=k_0}^{k-1} \alpha_i^2}\right) \le e^{-\epsilon}.
\end{align}

Next, we want to bound $\max_{s \in [k_0, k-1]} \sum_{i=k_0}^{s} \alpha_i^2 \|\bz_i\|^2$.

The following is the modification of the proof of \citep[Theorem 2.6, IV]{wainwright2019high}.

Multiplying both sides of the definition of sub-Gaussian random vectors \eqref{eq:sub-Gaussian} by $e^{-\frac{1}{2t} \|v\|^2 \hat{\sigma}^2}$ with $t \in (0,1)$ gives
\begin{align} \label{eq:left_eq_right}
  \bbE\left[e^{v^\top \bz - \frac{1}{2t} \|v\|^2 \hat{\sigma}^2 }\right] \le e^{-\frac{1}{2}(\frac{1}{t} -1 ) \hat{\sigma}^2 \|v\|^2}.
\end{align}

Integrating both sides over $v$ gives
\begin{align} \label{eq:right}
  \int e^{-\frac{1}{2}(\frac{1}{t} -1 ) \hat{\sigma}^2 \|v\|^2} dv = \frac{\left(2 \pi\right)^{\frac{n}{2}}}{\left((\frac{1}{t} -1) \hat{\sigma}^2\right)^{\frac{n}{2}}}
\end{align}
and
\begin{equation}\label{eq:left}
\begin{aligned}[b]
  \int e^{v^\top \bz -\frac{1}{2t} \hat{\sigma}^2 \|v\|^2} dz 
  &= e^{\frac{1}{2}\frac{t}{\hat{\sigma}^2} \|\bz\|^2} \int e^{-\frac{1}{2} \frac{\hat{\sigma}^2}{t} \|v- \frac{t}{\hat{\sigma}^2} \bz\|^2}dv \\
  & =  e^{\frac{1}{2}\frac{t}{\hat{\sigma}^2} \|\bz\|^2} \frac{\left(2 \pi\right)^{\frac{n}{2}}}{\left(\frac{\hat{\sigma}^2}{t}\right)^{\frac{n}{2}}}.
\end{aligned}
\end{equation}

Plugging \eqref{eq:right} and \eqref{eq:left} into \eqref{eq:left_eq_right}, we have for all $t \in (0,1)$, 
\begin{align*}
  \bbE\left[ e^{\frac{1}{2}\frac{t}{\hat{\sigma}^2} \|\bz\|^2}\right] \le \frac{1}{(1-t)^{\frac{n}{2}}}.
\end{align*}

Let $\lambda = \frac{1}{2} \frac{t}{\hat{\sigma}^2}$.
Then for $0 < \lambda < \frac{1}{2 \hat{\sigma}^2}$,
\begin{align*}
  \bbE\left[ e^{\lambda \|\bz\|^2}\right] \le \frac{1}{(1- 2\hat{\sigma}^2 \lambda)^{\frac{n}{2}}}.
\end{align*}

Let $g(\lambda) = \log \frac{1}{(1- 2\hat{\sigma}^2 \lambda)^{\frac{n}{2}}} = - \frac{n}{2} \log(1-2 \lambda \hat{\sigma}^2)$. 
Then, applying the Taylor expansion gives 
\begin{align*}
  g(\lambda) = \frac{n}{2} \sum_{k=1}^\infty \frac{1}{k !} (2 \lambda \hat{\sigma}^2)^k &= n \lambda \hat{\sigma}^2 + \sum_{k=2}^\infty \frac{1}{k !} (2 \lambda \hat{\sigma}^2)^k \\
  & \le  n \lambda \hat{\sigma}^2 + \frac{1}{2} \frac{(2 \lambda \hat{\sigma}^2)^2}{1 - (2 \lambda \hat{\sigma}^2)}.
\end{align*}

If $2 \lambda \hat{\sigma}^2 \le \frac{1}{2}$, then $\lambda \le \frac{1}{4 \hat{\sigma}^2}$.
Therefore, $g(\lambda) \le  n \lambda \hat{\sigma}^2 + 4 \lambda^2 \hat{\sigma}^4$ and 
\begin{align*}
  \bbE\left[ e^{\lambda \|\bz\|^2}\right] \le e^{ n \lambda \hat{\sigma}^2 + 4 \lambda^2 \hat{\sigma}^4}.
\end{align*}

If $0 < \lambda \le \frac{1}{4 \alpha_i^2 \hat{\sigma}^2}$ for all $i \in [k_0, k-1]$, then we can show that 
\begin{align*}
  M_{s}(\lambda) = e^{\sum_{i = k_0}^s \left(\lambda \alpha_i^2 \|\bz_i\|^2 - n \lambda \alpha_i^2 \hat{\sigma}^2 - 4 \lambda^2 \alpha_i^4 \hat{\sigma}^4\right)}
\end{align*}
is also supermartingale with $\bbE\left[ M_{s}(\lambda)\right] \le 1$ for all $s \in [k_0, k-1]$.

Similar to the process of getting \eqref{eq:lastline} and choosing the same $\epsilon$, we have 
\begin{align}
  &\bbP\left(\max_{s \in [k_0, k-1]} M_{s}(\lambda) \ge e^{\epsilon} \right) \le e^{-\epsilon} \bbE\left[ M_{k_0}(\lambda)\right] \le e^{- \epsilon} \nonumber \\
  \Rightarrow \: & \bbP \left( \max_{s \in [k_0, k-1]}  \sum_{i = k_0}^{s} \alpha_i^2 \|\bz_i\|^2 \ge \frac{\epsilon}{\lambda} + \lambda 4 \sum_{i=k_0}^{k-1} \alpha_i^4 \hat{\sigma}^4 + \sum_{i=k_0}^{k-1} \alpha_i^2 n \hat{\sigma}^2 \right) \le e^{- \epsilon} \label{eq:lastline_2}
\end{align}

We can choose $\lambda = \frac{1}{4 \hat{\sigma}^2 \max_{i \in [k_0, k-1] }\alpha_i^2}$ and plugging it into the RHS of the inequality inside the probability in \eqref{eq:lastline_2}. Then, the following holds:
\begin{align}
  \nonumber
  &\bbP \left( \max_{s \in [k_0, k-1]}  \sum_{i = k_0}^{s} \alpha_i^2 \|\bz_i\|^2 \ge  4 \hat{\sigma}^2 \epsilon  \max_{i \in [k_0, k-1]} \alpha_i^2 + \hat{\sigma}^2 \frac{1}{\max_{i \in [k_0, k-1]} \alpha_i^2} \sum_{i=k_0}^{k-1} \alpha_i^4  + \sum_{i=k_0}^{k-1} \alpha_i^2 n \hat{\sigma}^2 \right)  \le e^{- \epsilon} \\
  \Rightarrow \:&
  \bbP \left( \max_{s \in [k_0, k-1]}  \sum_{i = k_0}^{s} \alpha_i^2 \|\bz_i\|^2 \ge  4 \hat{\sigma}^2 \epsilon \max_{i \in [k_0, k-1]} \alpha_i^2 + \left(\hat{\sigma}^2 + n \hat{\sigma}^2\right)\sum_{i=k_0}^{k-1} \alpha_i^2  \right)  \le e^{- \epsilon}  \nonumber \\
  \Rightarrow \: &
  \bbP \left( \max_{s \in [k_0, k-1]}  \sum_{i = k_0}^{s} \alpha_i^2 \|\bz_i\|^2 \ge \left( 4 \hat{\sigma}^2 \epsilon + \hat{\sigma}^2 + n \hat{\sigma}^2\right) \sum_{i=k_0}^{k-1} \alpha_i^2\right)  \le e^{- \epsilon} . \label{eq:pr_event_B}
\end{align}
where the first arrow holds because $\sum_{i=k_0}^{k-1} \alpha_i^4 \le  \max_{j \in [k_0, k-1]} \alpha_j^2 \sum_{i=k_0}^{k-1} \alpha_i^2$.

The intersection of the respective complements of the events in \eqref{eq:pr_event_A} and \eqref{eq:pr_event_B} is the event that $\max_{i \in [k_0, k-1]} \|\bx_{\tau_i}^B - \bx_{\tau_i}^M\|$ is upper bounded by 
$$
\left( 2 \sqrt{2} \hat{\sigma} D \sqrt{\epsilon}\sqrt{ \sum_{j=k_0}^{k-1} \alpha_j^2}  + \left( 4 \hat{\sigma}^2  \epsilon + \hat{\sigma}^2  + n \hat{\sigma}^2\right)  \sum_{j=k_0}^{k-1} \alpha_j^2 \right)^{1/2}.
$$
Such an event occurs with probability $(1-e^{-\epsilon})^2$.

Further combining \eqref{eq:decomposed_x_A2M} and \eqref{eq:deterministic_x_A2B} completes the proof.

\end{proofof}

  In the following proof, we follow the notation in \citep{rockafellar2015convex}. Let $\gamma(x \vert \cX)$ denote the gauge function:
\begin{align*}
  \gamma(x \vert \cX) = \inf\{t >0 \vert x \in t \cX\}
\end{align*}
and let $\delta(x \vert \cX)$ be the support function:
\begin{align*}
  \delta(x \vert \cX) = \sup\{y^\top x \vert y \in \cX\}.
\end{align*}

\begin{proofof}{Lemma~\ref{lem:x_C2D}}

Applying  Lemma 2.2 (i) in \citep{tanaka1979stochastic} gives

\begin{align} \label{eq:x_C2D}
    \left\|\bx_{\tau_k}^C - \bx_{\tau_k}^D \right\|^2 &\le \left\|\by_{\tau_k}^C - \by_{\tau_k}^D \right\|^2 + 2 \int_{\tau_{k_0}}^{\tau_k} (\by_{\tau_k}^C - \by_{\tau_k}^D - \by_s^C + \by_s^D)^\top(\bv_s^D d \bmu^D(s) - 
    \bv_s^C d \bmu^C(s)) \nonumber \\
    & \le 2 \int_{\tau_{k_0}}^{\tau_k} (\by_s^C - \by_s^D)^\top \bv_s^C d \bmu(s)  \nonumber\\
    & \le 2 \int_{\tau_{k_0}}^{\tau_k} \gamma(\by_s^C - \by_s^D \vert \cX) \delta(\bv_s^C \vert \cX) d \bmu^C(s) \nonumber\\
    & \le 2 \sup_{s \in [\tau_{k_0}, \tau_k] }\gamma(\by_s^C - \by_s^D \vert \cX) \int_{\tau_{k_0}}^{\tau_k} \delta(\bv_s^C \vert \cX) d \bmu^C(s).
\end{align}
The second inequality is because $\by_s^D = \by_{\tau_k}^C$ for all $s \in [\tau_k, \tau_{k+1})$, $\bmu^D$ is supported on the discrete set $\{\tau_0,\tau_1,\tau_2,\cdots\}$ and the integrand is zero on this set . The third inequality uses the inequality $x^\top y \le \gamma(x\vert \cX) \delta(y \vert \cX)$ and the last inequality follows H\"{o}lder's inequality.

Since $\cX$ contains a ball of radius $r$ around the origin, we have $\gamma(x \vert \cX) \le r^{-1} \|x\|$. Then, the following holds
\begin{equation}\label{eq:max_gauge}
\begin{aligned}[b] 
  \sup_{s \in [\tau_{k_0}, \tau_k] } \gamma(\by_s^C - \by_s^D \vert \cX) &\le r^{-1} \sup_{s \in [\tau_{k_0}, \tau_k] } \|\by_s^C - \by_s^D\| \\
    &\le r^{-1} \max_{j \in [{k_0},k]} \int_{\tau_j}^{\tau_{j+1}} \left\|\nabla \bar{f}(\bx_s^C)\right\| ds \\
    &\le r^{-1}  u \max_{j \in [{k_0},k]} \alpha_j.
\end{aligned}
\end{equation}

To bound the integral in \eqref{eq:x_C2D}, we take the following derivative
\begin{align} \label{eq:support_fun}
    d \|\bx_t^C\|^2 = 2 (\bx_t^C)^\top \left( - \nabla \bar{f}(\bx_t^C)dt - \bv_t^C d \bmu^C(t)\right) \nonumber\\
    \Leftrightarrow 2 (\bx_t^C)^\top \bv_t^C d \bmu^C(t) = -2 (\bx_t^C)^\top \nabla \bar{f}(\bx_t^C)dt - d \|\bx_t^C\|^2 
\end{align}

  By construction, $(\bx_t^C)^\top \bv_t^C = \sup\{x^\top \bv_t^C \vert x \in \cX\} = \delta(\bv_t^C \vert \cX)$. Therefore, taking the integral of  \eqref{eq:support_fun} gives
  \begin{equation} \label{eq:support_integral}
\begin{aligned}[b]
    & 2 \int_{\tau_{k_0}}^{\tau_k} \delta(\bv_s^C \vert \cX) d \mu(s) = -2 \int_{\tau_{k_0}}^{\tau_k}(\bx_s^C)^\top \nabla \bar{f}(\bx_s^C) ds+ \|\bx_{\tau_{k_0}}^C\|^2-  \|\bx_{\tau_k}^C\|^2 \\
\Leftrightarrow &
    \int_{\tau_{k_0}}^{\tau_k} \delta(\bv_s^C \vert \cX) d \mu(s) = - \int_{\tau_{k_0}}^{\tau_k}(\bx_s^C)^\top \nabla \bar{f}(\bx_s^C) ds + \frac{1}{2}\|\bx_{\tau_{k_0}}^C\|^2-  \frac{1}{2}\|\bx_{\tau_k}^C\|^2  \\
    & \hspace{90pt}\le (\tau_k - \tau_{k_0} )D u  + D^2.
\end{aligned}
\end{equation}

Plugging \eqref{eq:max_gauge} and \eqref{eq:support_integral} into \eqref{eq:x_C2D}, we have 
\begin{align*}
    \left\|\bx_{\tau_k}^C - \bx_{\tau_k}^D \right\|^2 \le 2 r^{-1}  u \left((\tau_k-\tau_{k_0}) D u  + D^2\right) \max_{j \in [k_0,k]} \alpha_j
\end{align*}
which gives 
\begin{align*}
  \left\|\bx_{\tau_k}^C - \bx_{\tau_k}^D \right\| \le \sqrt{2 r^{-1}  u \left( (\tau_k-\tau_{k_0}) D u  + D^2\right) \max_{j \in [k_0,k]} \alpha_j}.
\end{align*}
\end{proofof}

  \begin{proofof}{Lemma~\ref{lem:x_C2C}}
    \begin{align*}
        \left\|\frac{d\bx_t^C}{dt}\right\| &= \left\|\Pi_{T_{\cX}(\bx_t^C)}\left(- \nabla \bar{f}(\bx_t^C) \right)\right\| \\
        & = \left\|\Pi_{T_{\cX}(\bx_t^C)}\left(- \nabla \bar{f}(\bx_t^C)\right) - \Pi_{{T_{\cX}(\bx_t^C)}}\left(0\right) \right\| \\
        & \le \left\|\nabla \bar{f}(\bx_t^C)\right\|
    \end{align*}
    where the first equality uses $0 \in T_{\cX}(x)$ and the inequality uses the non-expansiveness of convex projection.
    
    Therefore, 
    \begin{align*}
        \left\|\bx_t^C - \bx_{\tau_k}^C \right\| &= \left\|\int_{\tau_k}^t \Pi_{T_{\cX}(\bx_s^C)}\left(- \nabla \bar{f}(\bx_s^C)\right) ds\right\| \\
        & \le \alpha_k u.
    \end{align*}
  \end{proofof}

\begin{proofof}{Lemma~\ref{lem:x_C2D_continuous}}

  For $t \in [\tau_k, \tau_{k+1})$, the triangle inequality gives
  \begin{align*}
      \|\bx_{t}^C - \bx_t^D\| & \le \|\bx_{t}^C - \bx_{\tau_k}^C +  \bx_{\tau_k}^C - \bx_t^D \| \\
      & \le \|\bx_{t}^C - \bx_{\tau_k}^C \| + \|\bx_{\tau_k}^C - \bx_{\tau_k}^D \| 
  \end{align*}
  
  Plugging Lemma~\ref{lem:x_C2C} and Lemma~\ref{lem:x_C2D} gives the desired bound.
  \end{proofof} 
  
\begin{proofof}{Lemma~\ref{lem:x_M2D}}

  Define $\brho_t = \bx_t^M + \by_t^M - \by_{\tau_k}^M - (\bx_t^D + \by_t^C -\by_t^D)$ for all $t \in [\tau_k, \tau_{k+1})$. This gives $\brho_{\tau_k} = \bx_{\tau_k}^M - \bx_{\tau_k}^D$.
  
  Then calculate 
  \begin{equation} \label{eq:d_rho}
  \begin{aligned}
      {d \|\brho_t\|} &= \left(\frac{\brho_t}{\|\brho_t\|} \right)^\top d \brho_t \\
      &= \left(\frac{\brho_t}{\|\brho_t\|} \right)^\top \left( \nabla \bar{f}(\bx_t^C) - \nabla \bar{f} (\bx_t^M)\right) dt \\
      & \le \left\| \nabla \bar{f}(\bx_t^C) - \nabla \bar{f}(\bx_t^M) \right\| dt \\
      &\le \ell \left\|\bx_t^M - \bx_t^C \right\| dt \\
      & \le \ell \left(\left\|\bx_t^M - \bx_t^D \right\| +\left\|\bx_t^D - \bx_t^C \right\| \right) dt
  \end{aligned}
\end{equation}
  where the second inequality is because $\nabla \bar{f}(x)$ is $\ell$-Lipschitz. 
  
  Taking the integral gives
  \begin{align*}
    \MoveEqLeft
      \|\brho_t\|  = \|\brho_{\tau_k}\| + \int_{\tau_k}^t d \|\brho_s\| \\
      & = \|\brho_{\tau_k}\|  + \lim_{\epsilon \downarrow 0 } \int_{\tau_k}^t  \indic{(\|\brho_s\| \ge \epsilon)} d \|\brho_s\| \\
      & \le \|\brho_{\tau_k}\|  + \lim_{\epsilon \downarrow 0 } \int_{\tau_k}^t  \indic{(\|\brho_s\| \ge \epsilon)} \ell  \left(\left\|\bx_s^M - \bx_s^D \right\| + \left\|\bx_s^D - \bx_s^C \right\| \right) ds \\
      & =  \|\brho_{\tau_k}\|  + \int_{\tau_k}^t  \ell  \left\|\bx_s^M - \bx_s^D \right\|ds + \int_{\tau_k}^t \ell  \left\|\bx_s^D - \bx_s^C \right\|  ds
  \end{align*}
  where the second equality is from Lemma 20 in \citep{lamperski2021projected} and the inequality uses \eqref{eq:d_rho}.
  
  Setting $t = \tau_{k+1}$ gives
  \begin{align}
      \|\brho_{\tau_{k+1}}\| \le (1+ \ell \alpha_k) \|\brho_{\tau_k}\| + \ell \int_{\tau_k}^{\tau_{k+1}} \left\|\bx_s^C - \bx_s^D \right\| ds.
  \end{align}
  
 Using the assumption that $\brho_{k_0} = \bx_{k_0}^M - \bx_{k_0}^D = 0$ and iterating gives
  \begin{align*}
      \|\brho_{\tau_k}\| &\le \sum_{i=k_0}^{k-1} \Pi_{j=i+1}^{k-1} (1 + \ell \alpha_j) \ell \int_{\tau_{i}}^{\tau_{i+1}} \left\|\bx_s^C - \bx_s^D \right\| ds \\
      & \le \sum_{i=k_0}^{k-1} \Pi_{j=i+1}^{k-1} (1 + \ell \alpha_j) \ell \alpha_{i}  \left(\max_{s \in [i, i+1]} \alpha_{s} u  +  \sqrt{2 r^{-1}  u \left((\tau_{i+1} -\tau_{k_0}) D u  + D^2 \right) \max_{j \in [k_0,i+1]} \alpha_j  }\right)\\
      & \le   \sum_{i=k_0}^{k-1} e^{\ell (\tau_{k} -\tau_{i+1})}\ell \alpha_{i} \left(\max_{s \in [i, i+1]} \alpha_{s} u  +   \sqrt{2 r^{-1}  u \left((\tau_{i+1} -\tau_{k_0}) D u  + D^2 \right) \max_{j \in [k_0,i+1]} \alpha_j  }\right) \\
      & =   \ell e^{\ell \tau_{k}} \sum_{i=k_0}^{k-1}  e^{-\ell \tau_{i+1}} \alpha_{i}  \left(\max_{s \in [i, i+1]} \alpha_{s} u  +  \sqrt{2r^{-1}  u \left((\tau_{i+1} -\tau_{k_0}) D u  + D^2 \right) \max_{j \in [k_0,i+1]} \alpha_j  }\right) \\
      & \le \ell e^{\ell \tau_{k}} \sum_{i=k_0}^{k-1} \int_{\tau_{i}}^{\tau_{i+1}} e^{-\ell w} dw \left(\max_{s \in [i, i+1]} \alpha_{s} u  +  \sqrt{2 r^{-1}  u \left((\tau_{i+1} -\tau_{k_0}) D u  + D^2 \right) \max_{j \in [k_0,i+1]} \alpha_j  }\right)\\
      & \le \ell e^{\ell \tau_{k}} \int_{\tau_{k_0}}^{\tau_k} e^{-\ell w} dw \left(\max_{s \in [k_0, k]} \alpha_{s} u  +  \sqrt{2 r^{-1}  u \left((\tau_{k}-\tau_{k_0}) D u  + D^2 \right)  \max_{j \in [k_0,k]} \alpha_j  }\right) \\
      & \le \ell e^{\ell \tau_{k}} \frac{1}{\ell} (e^{-\ell \tau_{k_0}}- e^{-\ell \tau_k}) \left(\max_{s \in [k_0, k]} \alpha_{s} u  +\sqrt{2 r^{-1}  u \left((\tau_{k}-\tau_{k_0}) D u  + D^2 \right)  \max_{j \in [k_0,k]} \alpha_j  }\right) \\
      & \le (e^{\ell (\tau_{k}-\tau_{k_0})} -1) \left(\max_{s \in [k_0, k]} \alpha_{s} u  +  \sqrt{2 r^{-1}  u \left((\tau_{k}-\tau_{k_0}) D u  + D^2 \right)  \max_{j \in [k_0,k]} \alpha_j  }\right) \\
      & \le (e^{\ell (\tau_{k}-\tau_{k_0})} -1) \max_{s \in [k_0, k]} \sqrt{\alpha_{s}} \left( u+ \sqrt{2 r^{-1}  u \left((\tau_{k}-\tau_{k_0}) D u  + D^2 \right)  }\right)
  \end{align*}
  where the second inequality uses Lemma~\ref{lem:x_C2D_continuous} and the last inequality uses that $\alpha_s \le \frac{1}{2}$ for all $s \in \bbN$.

  
  \end{proofof}

  \section{Supporting Results on Variational Geometry} \label{app:variational_geometry}
  The following lemmas are standard in the field of optimization and variational analysis. We present the proofs to support the results in the main paper.
  
  \begin{lemma} \label{lem:projection_characteristic}
    For any $x \in \bbR^n$ and convex set $\cX$, $y^* = \Pi_{\cX}(x)$ iff $x - y^* \in \cN_{\cX}(y^*)$ and $y^* \in \cX$.
    
  \end{lemma}
  \begin{proof}
    First, the definition of the convex projection is equivalent to
    \begin{align*}
      \Pi_{\cX}(x) = \argmin_{y \in \cX} \frac{1}{2} \|y-x\|^2.
    \end{align*}
    Set $f(y) = \frac{1}{2} \|y-x\|^2$ which is strongly convex thus has a unique minimizer.
  
    ($\Rightarrow$)
  
     Let $y^*$ be the minimizer of $f$, i.e. $y^* = \Pi_{\cX}(x)$. From the necessary optimality condition, we have $ - \nabla f(y^*) \in \cN_{\cX}(y^*)$, i.e. $x - y^* \in \cN_{\cX}(y^*)$. 
  
    ($\Leftarrow$)

    Let $y^* \in \cX$ and $x - y^* \in \cN_{\cX}(y^*)$.
  
    From the definition of normal cone, $x - y^* \in \cN_{\cX}(y^*)$ $\Leftrightarrow$ $\left< x - y^*,  y - y^* \right> \le 0, \; \forall y \in \cX$. Besides,
    \begin{align*}
      \|x - y\|^2 - \|x -y^*\|^2 &= \|x - y^* + y^* -y\|^2 -\|x-y^*\|^2 \\
      & = \|x - y^*\|^2 + \|y^* -y\|^2 +2 (x - y^*)^\top (y^* -y) -\|x-y^*\|^2  \\
      & \ge 0
    \end{align*}
    which implies that  $y^*$ is the minimizer of $f$, i.e. $y^* = \Pi_{\cX}(x)$.
  \end{proof}
  
  \begin{lemma} \label{lem:proj_tangent_cone_2}
    For all $x  \in \cX$, $g \in \bbR^n$, we have 
    \begin{align*}
      g^\top \Pi_{T_{\cX}(x)}(g) = \|\Pi_{T_{\cX}(x)}(g)\|^2
    \end{align*}
  \end{lemma}
  \begin{proof}
    It suffices to show that $\left(g - \Pi_{T_{\cX}(x)}(g)\right)^\top \Pi_{T_{\cX}(x)}(g) =0 $.
    

    Firstly, from Lemma~\ref{lem:projection_characteristic}, we have $g - \Pi_{T_{\cX}(x)}(g) \in \cN_{T_{\cX}(x)}(\Pi_{T_{\cX}(x)}(g))$, i.e. 
    \begin{align} \label{eq:orthogonality}
      \left(g - \Pi_{T_{\cX}(x)}(g)\right)^\top \Pi_{T_{\cX}(x)}(g) \ge \left(g - \Pi_{T_{\cX}(x)}(g)\right)^\top  y,  \;\forall y \in T_{\cX}(x).
    \end{align}
  
    For notation simplicity, set $\phi = g - \Pi_{T_{\cX}(x)}(g)$ for the analysis below. 
  
    Note that $0 \in T_{\cX}(x)$, then we have  $\phi^\top \Pi_{T_{\cX}(x)}(g) \ge 0$. Furthermore, from the definition of tangent cone, if $y \in T_{\cX}(x)$, then $t y \in  T_{\cX}(x) $ for all $t \ge 0$. 
    For the sake of contradiction, suppose $\phi^\top y > 0 $. Then, there exists $t >0$, such that $\phi^\top t y \ge \phi^\top \Pi_{T_{\cX}(x)}(g) $, which contradicts \eqref{eq:orthogonality}. Therefore, we conclude that $\phi^\top y \le 0 $, which further implies that $\phi^\top \Pi_{T_{\cX}(x)}(g) \le 0$ since $\Pi_{T_{\cX}(x)}(g) \in T_{\cX}(x)$. Therefore, we have $\phi^\top \Pi_{T_{\cX}(x)}(g) =0$ as desired.
  \end{proof}
  
  The following lemma is a special case of the Moreau decomposition, and enables us to use the Skorokhod problem framework. See \cite{hiriart2004fundamentals}. 
  \begin{lemma} \label{lem:proj_tangent_normal}
  For all $x \in \cX$, $g \in \bbR^n$,  the following holds
  \begin{align} \label{eq:equi_tangent_normal}
    \Pi_{T_{\cX}(x)}(g) = g - \Pi_{\cN_{\cX}(x)}(g).
  \end{align}
  \end{lemma}

  \section{Background on the Skorokhod Problem} \label{app:skorokhod}

This appendix presents background on the Skorokhod problem needed for the paper.


The Skorokhod problem is a classical framework for constraining stochastic processes to remain in a set. It is a useful tool to analyze projection-based algorithms in continous time. 

Let $\cX$ be a convex subset of $\bbR^n$ with non-empty interior. Let $y: [0, \infty) \rightarrow \bbR^n$ be a trajectory which is right-continous with left limits and has $y_0 \in \cK$. For each $x \in \bbR^n$, let $\cN_{\cX}$ be the normal cone at $x$. Then the functions $x_t$ and $\phi_t$ solve the \textit{Skorokhod problem} for $y_t$ if the following conditions hold:
\begin{itemize}
 \item $x_t = y_t + \phi_t \in \cX$ for all $t \in [0,T)$.
 \item The function $\phi$ has the form $\phi_t = - \int_0^t v_s d \mu(s)$, where $\|v_s\| \in\{0,1\}$ and $v_s \in \cN_{\cX}(x_s)$ for all $s \in [0,T)$, while the measure, $\mu$, satisfies $\mu([0, T)) < \infty$ for any $T >0$.
\end{itemize}
 It is shown in \citep{tanaka1979stochastic} that a solution exists and is unique when $y$ is riht-continuous with left limits  and $\cX$ is convex. The existence and uniqueness of the solution implies that we can view the Skorokhod solution as a mapping: $x = \cS(y)$. And we are often interested in $x_t$, thus we will call $x_t$ as the solution of the Skorokhod problem corresponding to $y_t$.

 In the following, we present the connection between Skorokhod problems and projected algorithms assuming $y_t$ is piecewise constant. Specifically, assuming that $0 = \tau_0 < \tau_1 <\cdots <\tau_{N-1} \le T$ are the jump points of $y_t$, and let $S_k = [\tau_k, \tau_{k+1})$ for $k < N -1$ and $S_{N-1} = [\tau_{N-1}, T]$. Then $y_t$ can be represented as 
 \begin{align*}
  y_t = \sum_{k=0}^{N-1} y_{\tau_k} \indic_{S_k}(t).
 \end{align*}

 Then, the solution of the Skorokhod problem has the form 
 \begin{align*}
    x_{\tau_{k+1}} = \Pi_{\cX}(x_{\tau_k} + y_{\tau_{k+1} }- y_{\tau_k}).
 \end{align*}



%
%
%

\end{document}